\documentclass[reqno]{amsart}

\usepackage{color}

\newtheorem{thm}{Theorem}[section]
\newtheorem{lem}[thm]{Lemma}
\newtheorem{cor}[thm]{Corollary}
\newtheorem{prop}[thm]{Proposition}
\newtheorem{rems}[thm]{Remarks}
\newtheorem{rem}[thm]{Remark}

\numberwithin{equation}{section}

\newcommand{\R}{\mathbb{R}}

\newcommand{\mS}{\mathbb{S}}
\newcommand{\U}{\mathbb{U}}

\newcommand{\A}{\mathbb{A}}
\newcommand{\B}{\mathbb{B}}

\newcommand{\E}{\mathbb{E}}
\newcommand{\Y}{\mathbb{Y}}
\newcommand{\V}{\mathbb{V}}

\newcommand{\ml}{\mathcal{L}}

\newcommand{\ve}{\varepsilon}
\newcommand{\rd}{\mathrm{d}}
\newcommand{\dom}{\mathrm{dom}}

\newcommand{\dhr}{\mathrel{\lhook\joinrel\relbar\kern-.8ex\joinrel\lhook\joinrel\rightarrow}}

\begin{document}


\title[Age-Structured Diffusive Populations]{An Evolution System for a Class of Age-Structured Diffusive Population Equations}

\author{Christoph Walker}
\email{walker@ifam.uni-hannover.de}
\address{Leibniz Universit\"at Hannover\\ Institut f\" ur Angewandte Mathematik \\ Welfengarten 1 \\ D--30167 Hannover\\ Germany}
\date{\today}

\begin{abstract}
Kato's theory on the construction of strongly continuous evolution systems associated with hyperbolic equations is applied to the linear equation describing an age-structured population that is subject to time-dependent diffusion. The evolution system is used to provide conditions for the well-posedness of the corresponding quasilinear equation.
\end{abstract}


\maketitle

\section{Introduction and Main Results}

\noindent Problems of the form
\begin{subequations}\label{QP} 
\begin{align}
\partial_t u+ \partial_au \, &=     A(u,t,a)u  \,, \qquad t\in (0,T]\, ,\quad a\in (0,a_m)\, ,\label{Q1}\\ 
u(t,0)&=\int_0^{a_m}b(a)\, u(t,a)\, \rd a\,, \qquad t\in (0,T]\, ,\label{Q2} \\
u(0,a)&=  \phi(a)\,, \qquad a\in (0,a_m)\,,\label{Q3a}
\end{align}
\end{subequations}
arise as an abstract formulation of  the evolution of 
an age- and spatially structured population with density  
$$
u=u(t,a):[0,T]\times [0,a_m)\rightarrow E_0\,.
$$ 
The age of individuals is denoted by $a\in [0,a_m)$ with
 maximal age $a_m\in (0,\infty]$. The spatial movement of individuals is described by the (usually: differential) operators
$$
A(u,t,a): E_1\subset E_0\rightarrow E_0
$$ 
which, in turn, may be influenced by the population itself.
We assume that these operators are  generators of (analytic) semigroups on the Banach space $E_0$ with domains~$E_1$.
Age-dependent death processes are incorporated into the operators $A(u,t,a)$. Equation~\eqref{Q2} is the birth law with birth rate $b$, and $\phi$ in~\eqref{Q3a} is the initial population. We refer e.g. to \cite{GurtinMcCamy74,GurtinMcCamy81,WalkerDCDSA10,WalkerJEPE,WebbSpringer} and the references therein for more details and concrete examples of age-structured diffusive population models.

Note that~\eqref{QP} has a quasilinear structure. Establishing well-posedness results for~\eqref{QP} naturally requires a prior investigation of the corresponding linear problem
\begin{subequations}\label{P} 
\begin{align}
\partial_t u+ \partial_au \, &=     A(t,a)u  \,, \qquad t\in (0,T]\, ,\quad a\in (0,a_m)\, ,\label{1}\\ 
u(t,0)&=\int_0^{a_m}b(a)\, u(t,a)\, \rd a\,, \qquad t\in (0,T]\, ,\label{2} \\
u(0,a)&=  \phi(a)\,, \qquad a\in (0,a_m)\,,
\end{align}
\end{subequations}
with an operator $A=A(t,a)$ depending on time $t$ and possibly age $a$, but being independent of~$u$. Allowing for a time dependence is essential for the study of quasilinear problems, the solvability of the latter is then a consequence. It is worth pointing out that a simultaneous dependence of the operator $A$ on both variables~$t$ and~$a$ induces additional conceptual and technical difficulties making the problem more intricate as we shall see. This may be one of the reasons why there does not seem to be much literature on problem~\eqref{P} (and~\eqref{QP}) involving a dependence on both variables~$t$ and~$a$ in the operator $A=A(t,a)$. Even for the time-independent case $A=A(a)$ there is no broad literature, however, see ~\cite{Rhandi,RhandiSchnaubelt_DCDS99,Thieme_LN_89,Thieme_AMASA_96,ThiemeDCDS}. 
A time-dependence was actually included in \cite{RhandiSchnaubelt_DCDS99} though rather in the birth process $b$ than in the diffusion part $A$.
For the case of an operator $A$ independent of time and age the corresponding literature list is more extensive, see for instance the references in \cite{WalkerDCDSA10,WebbSpringer}. 

In this work we shall rely on the results of~\cite{WalkerIUMJ} where it was shown that the solution to~\eqref{P} involving time-independent operators $A=A(a)$ is given by a semigroup on the (biologically natural) phase space $\E_0:=L_1(J,E_0)$ (see~\eqref{100} below) with a generator that is fully determinable.
Based on this result we shall prove herein  that an evolution system  can be associated with the time-dependent problem~\eqref{P}. This is done by rather direct computations  verifying the conditions of Kato's seminal theory \cite{Kato70,Kato73} (see also \cite{Pazy}) on  evolution equations of hyperbolic type. Moreover, we derive additional regularity properties of the evolution system and provide conditions for the existence of mild solutions to the quasilinear problem~\eqref{QP}.





\subsection*{Assumptions and Notations}

We shall first consider the linear problem~\eqref{P} and list our assumptions for this case. Roughly speaking, we will assume that $A=A(t,a)$ induces for fixed time $t$ a (parabolic) evolution operator with respect to age $a$ satisfying certain properties uniformly with respect to time $t$.
 
Set $J:=[0,a_m]$ if $a_m<\infty$ and $J:=[0,\infty)$ if $a_m=\infty$. We let $E_0$ be a real Banach space ordered by a closed convex cone $E_0^+$ and 
assume  that  $E_1$ is a dense and continuously embedded subspace of $E_0$.  We fix $T>0$ and assume that there is $\rho>0$ such that
\begin{subequations}\label{A1}
\begin{equation}\label{A1a}
A\in  C\big([0,T],L_\infty\big(J,\ml(E_1,E_0)\big)\big)\,,\qquad A(t)\in  C^\rho\big(J,\mathcal{H}(E_1,E_0)\big)\,,\quad t\in [0,T]\,,
\end{equation}
where $\mathcal{H}(E_1,E_0)$ is the open subset of $\ml(E_1,E_0)$ consisting of generators of analytic semigroups on $E_0$ with domains $E_1$. It then follows from \cite[Corollary~II.4.4.2]{LQPP} that, for every $t\in [0,T]$ fixed, $A(t)\in  C^\rho\big(J,\mathcal{H}(E_1,E_0)\big)$  generates a parabolic evolution operator
$$
U_{A(t)}:\{(a,\sigma)\,;\, a\in J\,,\ 0\le \sigma\le a\}\to \ml(E_0)
$$
on $E_0$ with regularity subspace $E_1$ in the sense of \cite[Section II.2]{LQPP}.
We further assume that there are $M\ge 1$ and $\varpi\in\R$ such that, for fixed $t \in [0,T]$,
	\begin{equation}\label{EO}
	\|U_{A(t)}(a,\sigma)\|_{\ml(E_\ell)}+(a-\sigma)^\ell\,\|U_{A(t)}(a,\sigma)\|_{\ml(E_0,E_\ell)}\le M e^{\varpi (a-\sigma)}\,,\qquad a\in J\,,\quad 0\le \sigma\le a \,,
	\end{equation}
for $\ell=0,1$ and some $M\ge 1$ (if $a_m<\infty$ so that $J$ is compact, this is automatically satisfied, see~\cite[Lemma ~II.5.1.3]{LQPP}), and that
	\begin{equation}\label{A4}
	\text{if $a_m=\infty$, then $\varpi<0$ in \eqref{EO}}\,.
	\end{equation}
\end{subequations}
In order to associate with \eqref{P} an evolution system we shall further impose that the family $(U_{A(t)})_{t\in [0,T]}$ is {\it stable} in $E_\ell$ for $\ell\in\{0,1\}$ in the sense that  there are $M_\ell\ge 1$ and $\omega_\ell\in \R$ such that
\begin{equation}\label{A11}
\bigg\|\prod_{j=1}^{n}U_{A(t_j)}(a_j,\sigma_j)\bigg\|_{\ml(E_\ell)}\le M_\ell \exp\bigg\{\omega_\ell\sum_{j=1}^{n}(a_j-\sigma_j)\bigg\}\,,\quad 0\le\sigma_j\le a_j<a_m\,,
\end{equation}
for any finite sequence $0\le t_1\le t_2\le\ldots\le t_n\le T$ and $n=1,2,\ldots $.

The birth rate~$b$ is supposed to satisfy
\begin{equation}\label{A2}
b\in L_{\infty}\big(J,\ml(E_\ell)\big)\cap L_{1}\big(J,\ml(E_\ell)\big)\,, \quad \ell\in \{0,\vartheta,1\}  
\,,
\end{equation}
where $E_\vartheta:=(E_0,E_1)_\vartheta$ with parameter  $\vartheta\in (0,1)$ is a fixed interpolation space corresponding to an admissible interpolation functor $(\cdot,\cdot)_\vartheta$. Moreover, we assume that
\begin{equation}\label{A3}
b(a)U_{A(t)}(a,0)\in \ml(E_0) \ \text{is strongly positive for $a$ in a subset of $J$ of positive measure}\, .
\end{equation}
We set
$$
\|b\|_\ell:=\|b\|_{L_{\infty}(J,\ml(E_\ell))}\,, \quad \ell\in \{0,\vartheta,1\}  \,,
$$
and write
$$
A(t,a):=A(t)(a)\in\mathcal{H}(E_1,E_0)\,,\quad a\in J\,,\quad t\in [0,T]\,,
$$
in the following.\\

Before continuing with the main results some remarks are in order concerning the imposed assumptions.

\begin{rems}\label{R1}
{\bf (a)} That $A(t)=A(t,\cdot)$ induces for fixed $t\in [0,T]$ a parabolic evolution operator $U_{A(t)}$ on~$E_0$ with regularity subspace $E_1$ in the sense of \cite[Section II.2]{LQPP} means in particular that 
$$
v(a):=U_{A(t)}(a,\sigma)v^0\,,\quad a\in [\sigma,a_m)\,,
$$ 
is, for  given $\sigma\in [0,a_m)$ and $v^0\in E_0$, the unique solution
$$
v\in C([\sigma,a_m),E_0) \cap C^1((\sigma,a_m),E_0)\cap C((\sigma,a_m),E_1)
$$
to the Cauchy problem
$$
\partial_a v(a)=A(t,a) v(a)\,,\quad a\in (\sigma,a_m)\,,\qquad v(\sigma)=v^0\,.
$$
Moreover, given $w^0\in E_0$  and $f\in \E_0=L_1(J,E_0)$, the mild solution $w\in C(J,E_0)$ to
$$
\partial_a w(a)=A(t,a) w(a)+f(a)\,,\quad a\in (0,a_m)\,,\qquad w(0)=w^0\,,
$$
is given by
$$
w(a)= U_{A(t)}(a,0)w^0+\int_0^a U_{A(t)}(a,\sigma)\,f(\sigma)\,\rd \sigma\,,\quad a\in J\,.
$$

{\bf (b)}  The stability assumption~\eqref{A11} corresponds to uniform resolvent estimates for the operator~$A$ (see~\cite[Theorem~5.2.2]{Pazy}, \cite[Section~II.4.2]{LQPP}) and reduces to the stability condition introduced in  \cite[Definition 5.2.1, Theorem 5.2.2]{Pazy} when $A(t)$ is independent of age $a\in J$ so that 
$$
U_{A(t)}(a,\sigma)=e^{(a-\sigma)A(t)}\,,\quad 0\le \sigma\le a\,,
$$
is the analytic semigroup generated by $A(t)$.\\

{\bf (c)}  For instance, assumption~\eqref{A11} is satisfied  if $M=1$ in \eqref{EO}. In particular, it is satisfied if $A(t)$ is independent of age and generates a contraction semigroup for every $t\in [0,T]$.\\

{\bf (d)} If $A\in  C\big([0,T],BC^1\big(J,\mathcal{H}(E_1,E_0)\big)\big)$ is such that 
$$
\|e^{sA(t,a)}\|_{\ml(E_0)}\le e^{\omega s}\,,\quad s\ge 0\,, \quad (t,a)\in [0,T]\times J\,,
$$
for some $\omega\in\R$, then~\cite[Theorem~5.4.8]{Pazy} (and the uniqueness of evolution operators) implies
$$
\|U_{A(t)}(a,\sigma)\|_{\ml(E_0)}\le e^{\omega (a-\sigma)}\,,\quad 0\le \sigma\le a< a_m\,, \quad t\in [0,T]\,,
$$
and hence assumption~\eqref{A11} is satisfied for $\ell=0$.\\

{\bf (e)} Suppose that $(Q(t))_{t\in [0,T]}$ is a familiy of isomorphisms from $E_1$ to $E_0$ such that 
$$
Q(t)A(t,\cdot)Q(t)^{-1}=A(t,\cdot)\,,\qquad \|Q(t)\|_{\ml(E_1,E_0)}+\|Q(t)^{-1}\|_{\ml(E_0,E_1)}\le C
$$ 
for every $t\in [0,T]$ and such that $Q:[0,T]\to\ml(E_1,E_0)$ is of bounded variation. If~\eqref{A11} is satisfied for $\ell=0$, then it is satisfied also for $\ell=1$. This follows exactly as in \cite[Theorem~5.2.4]{Pazy}. For instance, if $Q\in C^1([0,T],\mathrm{Isom}(E_1,E_0))$, if $Q(t)$ commutes with $A(t,\cdot)$, and if~\eqref{A11} is satisfied for $\ell=0$, then~\eqref{A11} is satisfied also for $\ell=1$. In particular, analogously to \cite[Theorem~5.4.8]{Pazy}, if $A\in C^1([0,T],\mathcal{H}(E_1,E_0))$ is age-independent and~\eqref{A11} is satisfied for $\ell=0$, one may take $Q(t)=\lambda_0-A(t)$ for some real $\lambda_0$ sufficiently large (in the resolvent sets of $A(t)$) and conclude that~\eqref{A11} is satisfied also for~$\ell=1$. \\

{\bf (f)}  Assumption~\eqref{A3} is required to cite easily the results from~\cite{WalkerIUMJ} on the semigroup associated with time-independent operators $A=A(a)$, see Proposition~\ref{T1} below.
\end{rems}

\subsection*{Main Results}

As mentioned previously, we first focus on the linear Cauchy problem~\eqref{P}. We begin by 
fixing $t\in [0,T]$ and considering for $v=v(s,a)$ the linear problem 
\begin{subequations}\label{Pt} 
\begin{align}
\partial_s v+ \partial_av \, &=     A(t,a)v \,, \qquad s>0\, ,\quad a\in (0,a_m)\, ,\label{1t}\\ 
v(s,0)&=\int_0^{a_m}b(a)\, v(s,a)\, \rd a\,, \qquad s>0\, ,\label{2t} \\
v(0,a)&=  \phi(a)\,, \qquad a\in (0,a_m)\,.
\end{align}
\end{subequations}
In \cite{WalkerIUMJ} it was shown that the solution to~\eqref{Pt} is determined by a strongly continuous semigroup $(\mS_t(s))_{s\ge 0}$ in $\E_0$, given by
\begin{subequations}\label{100}
  \begin{equation}\label{100a}
     \big[\mS_t(s) \phi\big](a)\, :=\, \left\{ \begin{aligned}
    &U_{A(t)}(a,a-s)\, \phi(a-s)\, ,& &   a\in J\,,\ 0\le s< a\, ,\\
    & U_{A(t)}(a,0)\, B_\phi^t(s-a)\, ,& &  a\in J\, ,\ s\ge a\, ,
    \end{aligned}
   \right.
    \end{equation}
for $\phi\in\E_0$, where $B_\phi^t$ is the unique solution to the Volterra equation 
  \begin{equation}\label{500}
\begin{split}
    B_\phi^t(s)\, &=\, \int_0^s \chi(a)\, b(a)\, U_{A(t)}(a,0)\, B_\phi^t(s-a)\, \rd
    a\,\\
& \qquad +\, \int_s^{a_m} \chi(a)\, b(a)\, U_{A(t)}(a,a-s)\, \phi(a-s)\, \rd a\, ,\quad
    s\ge 0\, ,
\end{split}
   \end{equation}
\end{subequations}		
with $\chi$ denoting the characteristic function of the interval $(0,a_m)$. That is, $B_\phi^t$ is such that
  \begin{equation}\label{6a}
   B_\phi^t(s)= \int_0^{a_m} b(a) \big[\mS_t(s)\phi\big](a)\, \rd a\, ,\quad s\ge 0\,.
    \end{equation}

We recall the main properties of the semigroup $(\mS_t(s))_{s\ge 0}$ in the next proposition (for more details and further properties regarding regularity, compactness, and asynchronous exponential growth we refer to~\cite{WalkerIUMJ}).

\begin{prop}{\bf \cite[Theorem~1.2, Theorem~1.4]{WalkerIUMJ}}\label{T1}
Suppose \eqref{A1}, \eqref{A2}, and \eqref{A3}. Let \mbox{$t\in [0,T]$} be fixed. There is
\begin{equation}\label{6ax}
[\phi\mapsto B_\phi^t]\in\ml \big(\E_0,C(\R^+,E_0)\big)
\end{equation} \vspace{1mm} 
such that $B_\phi^t=B_\phi^t(s)$ is the unique solution to~\eqref{500} for $\phi\in\E_0$. Moreover:

{\bf (a)} 
$(\mS_t(s))_{s\ge 0}$ defined in \eqref{100a} is a strongly continuous positive semigroup on $\E_0$ with
\begin{equation*} 
\|\mS_t(t)\|_{\ml(\E_0)}\le M_0\, e^{(\varpi +\|b\|_0 M_0)t}\,,\quad t\ge 0\,.
\end{equation*}

{\bf (b)} Let $\A(t)$ be the infinitesimal generator of the semigroup~$(\mS_t(s))_{s\ge 0}$. Then
$\psi\in \dom(\A(t))$ if and only if there exists $\phi\in \E_0$ such that $\psi\in C(J,E_0)\cap \E_0$ is the mild solution to
\begin{equation*}\label{psi}
\partial_a\psi = A(t,a)\psi +\phi(a)\,,\quad a\in J\,,\qquad \psi(0)=\int_0^{a_m} b(a) \psi(a)\,\rd a\,.
\end{equation*}
In this case, $\A(t) \psi = -\phi$. 
\end{prop}

As a consequence, we can reformulate problem~\eqref{P} as a linear Cauchy problem
\begin{equation}\label{CP}
u'(t)=\A(t) u(t)\,,\quad t\in [s,T]\,,\qquad u(s)=\phi\,,
\end{equation}
in the Banach space $\E_0$.
We aim at constructing an evolution system corresponding to this Cauchy problem. For this purpose we require more information on the (domain of the) generator $\A(t)$.

When replacing the space $\E_0=L_1(J,E_0)$ by $L_p(J,E_0)$ with $p\in (1,\infty)$ and assuming the operators $A(t,\cdot)$ to have maximal $L_p$-regularity, one can prove (see \cite{WalkerMOFM} for details) that the domain $\dom(\A(t))$  is given by
$$
\dom(\A(t))=\left\{\psi\in L_p(J,E_1)\cap W_p^1(J,E_0)\,;\,\psi(0)=\int_0^{a_m}b(a)\, \psi(a)\, \rd a\right\}\,.
$$
In particular, in this case the domain of the generator $\A(t)$ is independent of $t\in [0,T]$ which then allows one to apply directly the special case \cite[Theorem 5.4.8]{Pazy} of  Kato's theory on hyperbolic evolution equations in order to deduce the  existence of an evolution system in $L_p(J,E_0)$ associated with~\eqref{CP} (imposing time-differentiability on~$A(t,\cdot)$). 

However, in the biologically more relevant framework of $\E_0=L_1(J,E_0)$ and without the maximal regularity assumption on   $A(t,\cdot)$, the domains $\dom(\A(t))$ need not be independent of time and the construction of an evolution system becomes more involved. 
In this regard, we shall prove that
$$
\mathbb{Y}:=\left\{\psi\in \E_1\cap W_1^1(J,E_0)\,;\,\psi(0)=\int_0^{a_m}b(a)\, \psi(a)\, \rd a\right\}
$$
with $\E_1:=L_1(J,E_1)$ is for every $t\in [0,T]$ a core for $\A(t)$; that is, a  dense  subspace of $D(\A(t))$ (i.e. the domain $\dom(\A(t))$ endowed with its graph norm), see Lemma~\ref{L1a} below. In particular, $\Y$ is dense in~$\E_0$ and 
$$
\A(t)\psi=-\partial_a\psi+A(t,\cdot)\psi\,,\quad \psi\in\Y\,.
$$
We shall further prove  that $\mathbb{Y}$ is {\it $\A(t)$-admissible} for $t\in [0,T]$, i.e. that  $\mathbb{Y}$ is an invariant subspace of $\mS_t(s)$ (see Lemma~\ref{L1}) and the restriction $\big(\mS_t(s)\vert_{\mathbb{Y}}\big)_{s\ge 0}$ of $\big(\mS_t(s)\big)_{s\ge 0}$ to $\Y$ is a strongly continuous semigroup on $\mathbb{Y}$ (see Corollary~\ref{C1}). 
Moreover, relying on~\eqref{A11} we prove that the family $(\A(t))_{t\in [0,T]}$  is {\it stable} in $\E_0$ and in $\E_1$ (see Proposition~\ref{P2}). Unfortunately,
 the family $(\A_\Y(t))_{t\in [0,T]}$  of the parts of $\A(t)$ in $\mathbb{Y}$ does not seem to be stable in $\Y$ (unless the operators $A=A(t)$ are independent of $a\in J$) so that a significant condition required to apply Kato's theory is lacking (see condition $(H_2)$ in \cite[Theorem~5.3.1]{Pazy}). Yet we shall see that the stability in $\E_1$ (and $\E_0$) is sufficient to copy Kato's construction of evolution systems in form of the proof of \cite[Theorem~5.3.1]{Pazy} almost verbatim. This relies on the observation that the right-hand side of the identity 
$$
\big(\A(t_1)-\A(t_2)\big)\psi=\big(A(t_1,\cdot)-A(t_2,\cdot)\big)\psi\,,\quad \psi\in\Y\,,\quad t_1, t_2\in [0,T]\,,
$$
is well-defined even if only $\psi\in\E_1$.
Our main result then reads:

\begin{thm}\label{T2}
Suppose \eqref{A1}, \eqref{A11}, \eqref{A2}, and \eqref{A3}. Then, there is a unique evolution system $(\U_{\A}(t,s))_{0\le s\le t\le T}$ in $\E_0$ associated with~\eqref{CP} in the sense that $\U_{\A}(t,s)\in\ml(\E_0)$ satisfies
\begin{equation}\label{ES1}
\U_{\A}(s,s)=I\,,\qquad \U_{\A}(t,s)=\U_{\A}(t,r)\U_{\A}(r,s)\,,\quad 0\le s\le r\le t\le T\,,
\end{equation}
and
\begin{equation}\label{ES2}
(t,s)\mapsto\U_{\A}(t,s)\ \text{ is strongly continuous in $\E_0$ on }\ 0\le s\le t\le T\,.
\end{equation}
Moreover,
\begin{align}
\|\U_{\A}(t,s)\|_{\ml(\E_0)}\le M_0e^{(\omega_0+M_0\|b\|_0)(t-s)}\,,\quad 0\le s\le t\le T\,,\label{E1}\\
\frac{\partial^+}{\partial t}\U_{\A}(t,s)\phi\big\vert_{t=s}=\A(s)\phi\,,\quad \phi\in \mathbb{Y}\,,\quad
0\le s\le T\,,\label{E2}\\
\frac{\partial}{\partial s}\U_{\A}(t,s)\phi=-\U_{\A}(t,s)\A(s)\phi\,,\quad \phi\in \mathbb{Y}\,,\quad
0\le s\le t\le T\,.\label{E3}
\end{align}
In addition, given $\alpha\in [0,1)$  denote by $E_\alpha:=[E_0,E_1]_\alpha$ the complex interpolation space and set \mbox{$\E_\alpha:=L_1(J,E_\alpha)$}. Then $\U_{\A}(t,s)\in\ml(\E_\alpha)$ for $0\le s\le t\le T$ with
\begin{align}
\|\U_{\A}(t,s)\|_{\ml(\E_\alpha)}\le M_\alpha e^{\eta(t-s)}\,,\quad 0\le s\le t\le T\,,\label{E1x}
\end{align}
for some $M_\alpha\ge 1$ (depending on $M_0, M_1$, and $\alpha$) and $\eta:=\max\{\omega_0+M_0\|b\|_0\,,\,\omega_1+M_1\|b\|_1\}$. Moreover,
\begin{equation}\label{ES2alpha}
(t,s)\mapsto\U_{\A}(t,s)\ \text{ is strongly continuous in $\E_\alpha$ on }\ 0\le s\le t\le T\,.
\end{equation}
\end{thm}

The exponential bound~\eqref{E1} implies that the evolution system in $\E_0$ is expo\-nen\-tially stable provided that  \mbox{$\omega_0+M_0\|b\|_0<0$}. Also note that the second part of Theorem~\ref{T2} (together with~\eqref{ES1}) entails that $(\U_{\A}(t,s))_{0\le s\le t\le T}$ is an evolution system in~$\E_\alpha$ for every $\alpha\in [0,1)$.\\

In some cases one may derive additional information the evolution system:

\begin{cor}\label{C00}
Let $a_m<\infty$ and suppose  \eqref{A1}, \eqref{A11}, \eqref{A2}, and \eqref{A3}. Let  $p\in (1,\infty)$ and consider $\theta\in [0,1]$. Let $(\cdot,\cdot)_\theta$ be either the complex interpolation functor $[\cdot,\cdot]_\theta$ or the real interpolation functor $(\cdot,\cdot)_{\theta,p}$ and suppose that $E_\theta:=(E_0,E_1)_\theta$ is reflexive.
Then  there are constants $N_0\ge 1$ and $\xi_0\in\R$ (depending on $M_\ell$, $\omega_\ell$, $b$, and $\theta$) such that
\begin{equation}\label{emb1r}
\|\U_{\A}(t,s)\|_{\ml(L_p(J,E_\theta))}\le N_0 e^{\xi_0(t-s)}\,,\quad 0\le s\le t\le T\,.
\end{equation}
In particular, 
\begin{equation}\label{emb2r}
\|\U_{\A}(t,s)\|_{\ml(\Y,\E_\theta)}\le N_0 e^{\xi_0(t-s)} \,,\quad 0\le s\le t\le T\,,
\end{equation}
for every $\theta\in [0,1)$, where $\E_\theta=L_1(J,E_\theta)$.
\end{cor}

\begin{rem}
The complex interpolation space $[E_0,E_1]_{\theta}$ and the real interpolation space $(E_0,E_1)_{\theta,p}$ with $p\in (1,\infty)$ are reflexive if both~$E_0$ and~$E_1$ are reflexive (e.g. see \cite[Corollary 4.5.2, Theorem~3.8.1]{BerghLoefstroem}). For $\theta=0$ or $\theta=1$ only $E_0$ respectively $E_1$ need to be reflexive, of course. 
\end{rem}

Theorem~\ref{T2} implies that one can associate a mild solution with the (non-homegeneous) Cauchy problem
\begin{equation}\label{NCP}
u'(t)=\A(t) u(t)+f(t)\,,\quad t\in [0,T]\,,\qquad u(0)=\phi\,.
\end{equation}
Indeed, under the assumptions of Theorem~\ref{T2} and given $\phi\in\E_\alpha$ and $f\in L_1((0,T),\E_\alpha)$ for some $\alpha\in (0,1)$, the mild solution $u\in C([0,T],\E_\alpha)$ to~\eqref{NCP} is 
\begin{equation}\label{vdk}
u(t)=\U_{\A}(t,0)\phi+\int_0^t \U_{\A}(t,\sigma)\, f(\sigma)\,,\quad 0\le t\le T\,.
\end{equation}
Unfortunately, it is not know whether or not the evolution system $\U_{\A}(t,s)$ leaves $\Y$ invariant and is strongly continuous on $\Y$; the only information in this direction stems from~\eqref{emb2r}. Thus, a similar result as in \cite[Theorem~5.4.3]{Pazy} regarding the existence of regular $\Y$-valued (strong) solutions to~\eqref{NCP}, which would necessarily be given by formula~\eqref{vdk}, may not be available. Nevertheless, Theorem~\ref{T2} is the basis to provide the existence of mild solutions to quasilinear problems as shown next.

\subsection*{Quasilinear Case}

We may interpret the evolution equation~\eqref{QP}  as a quasilinear Cauchy problem
\begin{equation}\label{QCP}
u'(t)=\A\big(u(t),t\big) u(t)\,,\quad t\in [0,T]\,,\qquad u(0)=\phi\,,
\end{equation}
in $\E_0$. 
In order to provide a well-posedness result we consider $\phi\in\E_1$ and let $\mathcal{B}:=\bar\B_{\E_0}(\phi,r_0)$ be the closed ball in $\E_0$  centered at $\phi$ with some positive radius $r_0>0$. Suppose that 
\begin{subequations}\label{AA}
\begin{equation}\label{lip1}
A(v)=A(v,\cdot,\cdot)\in   C\big([0,T],L_\infty\big(J,\mathcal{L}(E_1,E)\big)\big)\,,\quad v\in\mathcal{B}\,,
\end{equation}
and there is $\rho>0$ with
\begin{equation}\label{lip2}
A(v,t,\cdot)\in   C^\rho\big(J,\mathcal{H}(E_1,E)\big)\,,\quad t\in [0,T]\,,\quad v\in\mathcal{B}\,.
\end{equation}
Moreover, assume there is $L>0$ with
\begin{equation}\label{lipp}
\| A(v_1,t,\cdot)-A(v_2,t,\cdot)\|_{L_\infty(J,\ml(E_1,E_0))}\le L\,\|v_1-v_2\|_{\E_0}\,,\quad v_1,v_2\in\mathcal{B}\,,\quad t\in [0,T]\,.
\end{equation}
Further suppose that for $\ell\in\{0,1\}$ there are $M_\ell\ge 1$ and $\omega_\ell\in \R$ such that
\begin{equation}\label{A11b}
\bigg\|\prod_{j=1}^{n}U_{A(v_j,t_j)}(a_j,\sigma_j)\bigg\|_{\ml(E_\ell)}\le M_\ell \exp\bigg\{\omega_\ell\sum_{j=1}^{n}(a_j-\sigma_j)\bigg\}\,,\quad 0\le\sigma_j\le a_j<a_m\,,
\end{equation}
for arbitrary $v_j\in \mathcal{B}$ and any finite sequence $0\le t_1\le t_2\le\ldots\le t_n\le T$ and $n=1,2,\ldots $.
Moreover, 
\begin{align}\label{A21}
\text{$[t\mapsto A(v,t)]$ satisfies \eqref{EO} with \eqref{A4} and \eqref{A3} for every $v\in\mathcal{B}$}\,.
\end{align}
\end{subequations}

We can state the following well-posedness result regarding the quasilinear problem~\eqref{QCP}:

\begin{thm}\label{T3}
Suppose \eqref{A2}. Let $\phi\in\E_1$ and assume~\eqref{AA}. Then, there are $T_\phi\in (0,T]$ and a unique mild solution $u\in C([0,T_\phi],\E_0)$ to the quasilinear problem~\eqref{QCP} with $u(t)\in \mathcal{B}$ for~$t\in [0,T_\phi]$, given by
$$
u(t)=\U_{\A(u)}(t,0)\phi\,,\quad t\in [0,T_\phi]\,,
$$
where $(\U_{\A(u)}(t,s))_{0\le s\le t\le T_\phi}$ is the unique evolution system generated by $t\mapsto \A\big(u(t),t\big)$. Moreover, $u\in C([0,T_\phi],\E_\alpha)$ for every $\alpha\in [0,1)$ with $\E_\alpha=L_1(J,[E_0,E_1]_\alpha)$.
\end{thm}

Based on Corollary~\ref{C00} we may prove a slightly different result:

\begin{cor}\label{CT3}
Let $a_m<\infty$ and fix $p\in (1,\infty)$ and $\theta\in [0,1]$. 
Let $(\cdot,\cdot)_\theta$ be either the complex interpolation functor $[\cdot,\cdot]_\theta$ or the real interpolation functor $(\cdot,\cdot)_{\theta,p}$ and suppose that $E_\theta:=(E_0,E_1)_\theta$ is reflexive.
Let $\phi\in \E_1\cap L_p(J,E_\theta)$. Then Theorem~\ref{T3} remains true if~$\mathcal{B}$ therein is everywhere replaced by 
$$
\mathcal{B}_0:=\big\{v\in \bar\B_{\E_0}(\phi,r_0)\,;\, v\in L_p(J,E_\theta)\,,\,\|v\|_{L_p(J,E_\theta)}\le (N_0+r_0)\|\phi\|_{L_p(J,E_\theta)}\big\}\subset \E_0
$$ 
for some $r_0>0$ and with $N_0$ given in Corollary~\ref{C00}. 
\end{cor}

Note that one may take $\theta=1$ in Corollary~\ref{CT3} (and Corollary~\ref{C00}) and that $ L_p(J,E_1)\subset  \E_1$ since $a_m<\infty$.

\subsection*{Paper Outline}
In Section~\ref{Sec2} we prove some auxiliary results. In particular, we show that for fixed $t\in [0,T]$,  the restrictions of the semigroups $(\mS_t(s))_{s\ge 0}$, defined in~\eqref{100}, to $\Y$ respectively to~$\E_1$ define again  strongly continuous semigroups. This then enables us to investigate in Section~\ref{Sec3} stability properties of the family of semigroups $(\mS_t)_{t\in [0,T]}$ in $\E_0$ and $\E_1$.  Section~\ref{Sec4} is dedicated to the proof of Theorem~\ref{T2}, where the construction of the evolution system $(\U_\A(t,s))_{0\le s\le t\le T}$ in $\E_0$ satisfying properties~\eqref{ES1}-\eqref{E3} follows along the lines of~\cite[Theorem~5.3.1]{Pazy}. The additional properties in $\E_\alpha$ are obtained by interpolation. In Section~\ref{Sec5} we give the details for Corollary~\ref{C00}. Finally, in Section~\ref{Sec6} we consider the quasilinear Cauchy problem~\eqref{QCP} and provide the proofs of Theorem~\ref{T3} and Corollary~\ref{CT3}.\\

Throughout the remainder of this paper we assume \eqref{A1},  \eqref{A2}, and \eqref{A3} (the stability assumption~\eqref{A11} is imposed later in Section~\ref{Sec3}).

\section{Admissibility Properties}\label{Sec2}

Recall that $E_1$ is continuously embedded into $E_0$, hence $\E_1=L_1(J,E_1)$ is continuously embedded into $\E_0=L_1(J,E_0)$. We may thus equip the Banach space
$$
\mathbb{Y}=\left\{\psi\in \E_1\cap W_1^1(J,E_0)\,;\,\psi(0)=\int_0^{a_m}b(a)\, \psi(a)\, \rd a\right\}
$$
 with the equivalent norm
\begin{equation}\label{normY}
\|\phi\|_{\Y}:=\|\phi\|_{\E_1}+\|\partial_a\phi\|_{\E_0}\,,\quad \phi\in\Y\,,
\end{equation}
for simplicity. It is also worth pointing out that, with the understanding 
$$
\big(A(t)\psi\big)(a):=A(t,a)\psi(a)\,,\quad a\in J\,,\quad t\in [0,T]\,,\quad \psi\in\E_1\,,
$$ 
we have 
\begin{equation}\label{s7ax}
\|A(t)\psi\|_{\E_0}\le \|A(t,\cdot)\|_{L_\infty(J,\ml(E_1,E_0))}\,\|\psi\|_{\E_1}\,,\quad \psi\in \E_1\,,
\end{equation}
for $t\in [0,T]$, so that we may interpret $A(t)=A(t,\cdot)\in\ml(\E_1,\E_0)$ with
\begin{equation}\label{normA}
\|A(t)\|_{\ml(\E_1,\E_0)}\le \|A(t,\cdot)\|_{L_\infty(J,\ml(E_1,E_0))}\,,\quad t\in [0,T]\,,
\end{equation}
in the following. Moreover,
\begin{equation}\label{contA}
\big(t\mapsto A(t)\big)\in C\big([0,T],\ml(\E_1,\E_0)\big)
\end{equation}
according to assumption~\eqref{A1a}.

\begin{lem}\label{L1a}
$\Y$ is  a dense subspace of $\E_0$ and of $D(\A(t))$ for $t\in [0,T]$. Moreover,
\begin{equation}\label{eq}
\A(t)\psi=-\partial_a\psi+A(t,\cdot)\psi\,,\quad \psi\in\Y\,,
\end{equation}
and \mbox{$\A\in C\big([0,T],\ml(\Y,\E_0)\big)$}.
\end{lem}

\begin{proof}
Fix $t\in [0,T]$. Since $\psi\in\Y$ implies $\phi:=\partial_a\psi-A(t,\cdot)\psi\in\E_0$, we deduce $\Y\subset \dom(\A(t))$ and~\eqref{eq} from Proposition~\ref{T1}.
Thus, using~\eqref{normA},
\begin{equation}\label{s5x}
\|\A(t)\psi\|_{\E_0}\le\|\partial_a\psi\|_{\E_0}+ \|A(t,\cdot)\|_{L_\infty(J,\ml(E_1,E_0))}\,\|\psi\|_{\E_1}\,,\quad \psi\in\Y\,,
\end{equation}
so that $\A(t)\in \ml(\Y,\E_0)$ and 
$$
\|\A(t_1)-\A(t_2)\|_{\ml(\Y,\E_0)}\le \|A(t_1,\cdot)-A(t_2,\cdot)\|_{L_\infty(J,\ml(E_1,E_0))}\,,\quad t_1, t_2\in [0,T]\,.
$$
Now, \eqref{A1a} entails $\A\in C\big([0,T],\ml(\Y,\E_0)\big)$. 
To prove that $\Y$ is dense in $D(\A(t))$ let \mbox{$\psi\in D(\A(t))$} and $\ve>0$ be arbitrary. Choose a real number $\lambda$ in the resolvent set of $\A(t)$ with $\lambda>\omega_*$, where  $\omega_*:=\max\{\omega_0,\omega_1\}$. Setting $\phi:=(\lambda-\A(t))\psi\in\E_0$ there is $\phi_\ve\in C_c(J,E_\vartheta)$ with 
$$
\|\phi_\ve-\phi||_{\E_0}\le \frac{\ve}{\|(\lambda-\A(t))^{-1}\|_{\ml(\E_0,D(\A(t)))}}\,.
$$
It then follows from Proposition~\ref{T1} and Remark~\ref{R1}~(a) that $\psi_\ve:=(\lambda-\A(t))^{-1}\phi_\ve\in D(\A(t))$ is given by
$$
\psi_\ve(a)=e^{-\lambda a}U_{A(t)}(a,0)\psi_\ve(0)+\int_0^ae^{-\lambda(a-\sigma)}U_{A(t)}(a,\sigma)\phi_\ve(\sigma)\,\rd \sigma\,,\quad a\in J\,,
$$
with
$$
\psi_\ve(0)=\int_0^{a_m}b(a)\, \psi_\ve(a)\, \rd a\,.
$$
In particular, $\psi_\ve(0)\in E_\vartheta$ since, using the regularizing effects~\eqref{EO} (and interpolation),
\begin{align*}
\int_0^{a_m}\|b(a)\, \psi_\ve(a)\|_{E_\vartheta}\, \rd a&
\le \|b\|_\vartheta \int_0^{a_m} e^{-\lambda a}\|U_{A(t)}(a,0)\|_{\ml(E_0,E_\vartheta)}\, \rd a\,\|\psi_\ve(0)\|_{E_0}\\
&\quad + \|b\|_\vartheta \int_0^{a_m}\|\phi_\ve(\sigma)\|_{E_\vartheta} \int_\sigma^{a_m}e^{-\lambda(a-\sigma)} \|U_{A(t)}(a,\sigma)\|_{\ml(E_\vartheta)}\,\rd a\,\rd\sigma\\
&\le c(\vartheta) \int_0^{a_m} a^{-\vartheta}\,e^{a (\omega_*-\lambda)}\, \rd a\,\|\psi_\ve(0)\|_{E_0}\\
&\quad + c(\vartheta) \int_0^{a_m}\|\phi_\ve(\sigma)\|_{E_\vartheta} \int_\sigma^{a_m} e^{(a-\sigma) (\omega_*-\lambda)}\,\rd a\,\rd\sigma
\end{align*}
with finite right-hand side since $\phi_\ve\in C_c(J,E_\vartheta)\subset L_1(J,E_\vartheta)$ and due to $\vartheta\in (0,1)$ and $\lambda>\omega_*$. A similar computation yields then $\psi_\ve\in \E_1$ since
\begin{align*}
\|\psi_\ve\|_{\E_1}& =\int_0^{a_m}\| \psi_\ve(a)\|_{E_1}\, \rd a\\
&
\le  \int_0^{a_m} e^{-\lambda a}\|U_{A(t)}(a,0)\|_{\ml(E_\vartheta,E_1)}\, \rd a\,\|\psi_\ve(0)\|_{E_\vartheta}\\
&\quad +  \int_0^{a_m}\|\phi_\ve(\sigma)\|_{E_\vartheta} \int_\sigma^{a_m}e^{-\lambda(a-\sigma)} \|U_{A(t)}(a,\sigma)\|_{\ml(E_\vartheta,E_1)}\,\rd a\,\rd\sigma\\
&\le c(\vartheta) \int_0^{a_m} a^{\vartheta-1}\,e^{a (\omega_*-\lambda)}\, \rd a\,\|\psi_\ve(0)\|_{E_\vartheta}\\
&\quad + c(\vartheta) \int_0^{a_m}\|\phi_\ve(\sigma)\|_{E_\vartheta} \int_\sigma^{a_m}(a-\sigma)^{\vartheta-1}\,e^{(a-\sigma) (\omega_*-\lambda)}\,\rd a\,\rd\sigma
\end{align*}
with finite right-hand side since $\psi_\ve(0)\in E_\vartheta$, $\phi_\ve\in C_c(J,E_\vartheta)$, and due to $\vartheta\in (0,1)$ and $\lambda>\omega_*$. This and~\eqref{normA} imply that
$$
\partial_a\psi_\ve= (A(t,\cdot)-\lambda)\psi_\ve+\phi_\ve \in\E_0
$$
and hence  $\psi_\ve\in \Y$. By construction,
$$
\|\psi_\ve-\psi||_{D(\A(t))}\le \|(\lambda-\A(t))^{-1}\|_{\ml(\E_0,D(\A(t)))}\,\|\phi_\ve-\phi||_{\E_0}\le \ve
$$
which yields the density of $\Y$ in $D(\A(t))$. Since the latter is dense in $\E_0$ (as the domain of the generator $\A(t)$), we also have that $\Y$ is dense in $\E_0$.
\end{proof}

We next show that $\mS_t(s)$ maps the space $\Y$ into itself for every $s\ge 0$ and $t\in [0,T]$.

\begin{lem}\label{L1}
Let $t\in [0,T]$. Then $\E_1$ and
$\Y$ are invariant under $\mS_t(s)$ for $s\ge 0$. More precisely,  $\mS_t(s)\vert_{\E_1}\in\ml(\E_1)$ and $\mS_t(s)\vert_{\Y}\in\ml(\Y)$.
\end{lem}

\begin{proof}
Fix $t\in [0,T]$ and recall the definition of $\mS_t$ and $B_\psi^t$ from~\eqref{100}.\\

\noindent{\bf (i)} Let $\psi\in\E_1$. We claim that then $B_\psi^t(s)\in E_1$ for $s\ge 0$. To this end, we use \eqref{EO} (and interpolation) along with~\eqref{A2} to get, for $\tau>0$ and $s\in [0,\tau]$,
\begin{align*}
\int_0^s &\chi(a)\|b(a)\|_{\ml(E_\vartheta)}\, \|U_{A(t)}(a,0)\|_{\ml(E_0,E_\vartheta)}\, \|B_\psi^t(s-a)\|_{E_0}\, \rd a\\
&\quad  + \int_s^{a_m} \chi(a)\| b(a)\|_{\ml(E_\vartheta)}\, \|U_{A(t)}(a,a-s)\|_{\ml(E_\vartheta)}\,\|\psi(a-s)\|_{E_\vartheta}\, \rd a\\
&\le M\,\|b\|_\vartheta\, \|B_\psi^t\|_{C([0,\tau],E_0)} \int_0^s a^{-\vartheta} e^{\varpi a}\, \rd a+ \|b\|_\vartheta \sup_{0\le \sigma\le a<a_m}\|U_{A(t)}(a,\sigma)\|_{\ml(E_\vartheta)}\, \|\psi\|_{L_1(J,E_\vartheta)}
    \end{align*}
and deduce from \eqref{A4} and \eqref{6ax} that the right-hand side of this estimate is finite. This, in turn, implies that the right-hand side of \eqref{500} belongs to $E_\vartheta$ and hence  $B_\psi^t(s)\in E_\vartheta$ with $\|B_\psi^t(s)\|_{E_\vartheta}\le c_\tau$ for $s\in [0,\tau]$ and  some $c_\tau>0$. This then similarly yields
\begin{align*}
\int_0^s & \chi(a)\|b(a)\|_{\ml(E_1)}\, \|U_{A(t)}(a,0)\|_{\ml(E_\vartheta,E_1)}\, \|B_\psi^t(s-a)\|_{E_\vartheta}\, \rd a\\
&\quad  + \int_s^{a_m} \chi(a)\| b(a)\|_{\ml(E_1)}\, \|U_{A(t)}(a,a-s)\|_{\ml(E_1)}\,\|\psi(a-s)\|_{E_1}\, \rd a\\
&\le M\,\|b\|_1\, c_\tau\int_0^s a^{\vartheta-1} e^{\varpi a}\, \rd a+ \|b\|_1 \sup_{0\le \sigma\le a<a_m}\|U_{A(t)}(a,\sigma)\|_{\ml(E_1)}\, \|\psi\|_{\E_1}
    \end{align*}
and we conclude analogously as above from \eqref{500} that indeed  $B_\psi^t(s)\in E_1$ for $s\ge 0$. Based on this we repeat the computation and get
\begin{align*}
\| B_\psi^t(s)\|_{E_1} &\le \int_0^s \chi(a)\|b(a)\|_{\ml(E_1)}\, \|U_{A(t)}(a,0)\|_{\ml(E_1)}\, \|B_\psi^t(s-a)\|_{E_1}\, \rd a\\
&\quad  + \int_s^{a_m} \chi(a)\| b(a)\|_{\ml(E_1)}\, \|U_{A(t)}(a,a-s)\|_{\ml(E_1)}\,\|\psi(a-s)\|_{E_1}\, \rd a\\
&\le \|b\|_1\,\sup_{0\le \sigma\le a<a_m}\|U_{A(t)}(a,\sigma)\|_{\ml(E_1)}\, \left(\int_0^s \|B_\psi^t(a)\|_{E_1}\, \rd a+\|\psi\|_{\E_1}\right)
    \end{align*}
for $s\ge 0$ so that \eqref{EO}, \eqref{A4} and  Gronwall's Inequality imply that
 \begin{equation*}
\| B_\psi^t(s)\|_{E_1}\le \|b\|_1\,\alpha_1\, e^{\|b\|_1\alpha_1 s}\,\|\psi\|_{\E_1}\,,\quad s\ge 0\,,
\end{equation*}
with
$$
\|b\|_1=\|b\|_{L_{\infty}(J,\ml(E_1))}\,,\qquad \alpha_1:=\sup_{0\le \sigma\le a<a_m}\|U_{A(t)}(a,\sigma)\|_{\ml(E_1)}<\infty\,.
$$
This then also entails 
 \begin{equation}\label{s7ab}
\mS_t(s)\psi\in\E_1\,,\quad s\ge 0\,,\quad \psi\in\E_1\,,
    \end{equation}
since similarly
\begin{align}
\| \mS_t(s)\psi\|_{\E_1} &\le \int_0^s  \|U_{A(t)}(a,0)\|_{\ml(E_1)}\, \|B_\psi^t(s-a)\|_{E_1}\, \rd a\nonumber\\
&\qquad  + \int_s^{a_m}  \|U_{A(t)}(a,a-s)\|_{\ml(E_1)}\,\|\psi(a-s)\|_{E_1}\, \rd a\nonumber\\
&\le \alpha_1^2\,\|b\|_1\, \|\psi\|_{\E_1}\int_0^s e^{\|b\|_1\alpha_1 \sigma}\,\rd \sigma +\alpha_1\|\psi\|_{\E_1}=\alpha_1\,e^{\|b\|_1\alpha_1 s}\,\|\psi\|_{\E_1}\,.\label{s7abcx}
    \end{align}
In particular, $\E_1$ is invariant under $\mS_t(s)$ for $s\ge 0$ and $\mS_t(s)\vert_{\E_1}\in\ml(\E_1)$. 
Recalling~\eqref{s7ax} we also deduce
 \begin{equation}\label{s7a}
A(t,\cdot)\mS_t(s)\psi\in\E_0\,,\quad s\ge 0\,,
    \end{equation}
with
 \begin{equation}\label{s7ac}
 \|A(t,\cdot)\mS_t(s)\psi\|_{\E_0}\le \|A(t,\cdot)\|_{L_\infty(J,\ml(E_1,E_0))}\,\alpha_1\,e^{\|b\|_1\alpha_1 s}\,\|\psi\|_{\E_1}\,,
 \quad s\ge 0\,.
    \end{equation}
\vspace{2mm}

\noindent {\bf (ii)} Let now $\psi\in\Y$. Then
\begin{equation}\label{s5}
\psi(0)=\int_0^{a_m}b(a)\, \psi(a)\, \rd a=B_\psi^t(0)\,,\qquad \A(t)\psi=A(t,\cdot)\psi-\partial_a\psi \in \E_0\,.
\end{equation}
Hence, differentiating \eqref{500} with respect to $s\ge 0$ yields
\begin{align}
    \partial_s B_\psi^t(s) &= \chi(s)b(s)U_{A(t)}(s,0) B_\psi^t(0)+\int_0^s \chi(a)\, b(a)\, U_{A(t)}(a,0)\, \partial_s B_\psi^t(s-a)\, \rd
    a \nonumber\\
& \quad -\chi(s)b(s)U_{A(t)}(s,0) \psi(0) + \int_s^{a_m} \chi(a)\, b(a)\, U_{A(t)}(a,a-s)\big[A(t,\cdot)\psi-\partial_a\psi\big](a-s)\, \rd a \nonumber\\
&=\int_0^s \chi(a)\, b(a)\, U_{A(t)}(a,0)\, \partial_s B_\psi^t(s-a)\, \rd
    a \nonumber\\
&\quad+ \int_s^{a_m} \chi(a)\, b(a)\, U_{A(t)}(a,a-s)\,\big(\A(t)\psi\big)(a-s)\, \rd a\,.\label{s11x}
    \end{align}
Therefore, by uniqueness,
\begin{equation}\label{s11}
\partial_s B_\psi^t(s)=B_{\A(t)\psi}^t(s)\,,\quad s\ge 0\,.
\end{equation} 
Note then that the first identity in \eqref{s5} implies $\big[\mS_t(s) \psi\big](s-)=\big[\mS_t(s) \psi\big](s+)$ if $s<a_m$, i.e. there is no jump across the diagonal $s=a$. Let $a\in J$. If $0\le s< a$, then
\begin{align*}
\partial_a\big[\mS_t(s) \psi\big](a)&=A(t,a)\,U_{A(t)}(a,a-s)\, \psi(a-s)\\
&\quad-U_{A(t)}(a,a-s)\, \big[A(t,a-s)\, \psi(a-s)- \partial_a\psi(a-s)\big]\\
&=A(t,a)\big[\mS_t(s) \psi\big](a) -\big[\mS_t(s)\big(\A(t)\psi\big)\big](a)\,.
\end{align*}
If $s\ge a$, then, using \eqref{s11},
\begin{align*}
\partial_a\big[\mS_t(s) \psi\big](a)&= A(t,a)U_{A(t)}(a,0)\, B_\psi^t(s-a)-U_{A(t)}(a,0)\, \partial_s B_\psi^t(s-a)\\
&= A(t,a)U_{A(t)}(a,0)\, B_\psi^t(s-a)-U_{A(t)}(a,0)\, B_{\A(t)\psi}^t(s-a)\\
&=A(t,a)\big[\mS_t(s) \psi\big](a) -\big[\mS_t(s)\big(\A(t)\psi\big)\big](a)\,.
\end{align*}
Gathering the two preceding identities and recalling~\eqref{s7a} and~\eqref{s5} yields
 \begin{equation}\label{s8a}
\partial_a\big[\mS_t(s) \psi\big]=A(t,\cdot) \mS_t(s) \psi -\mS_t(s)\big(\A(t)\psi\big) \in \E_0\,.
    \end{equation}
It now follows from \eqref{s8a}, \eqref{s7ab},  \eqref{6a}, and \eqref{100a} that $\mS_t(s) \psi\in \Y$.
 Consequently, $\Y$ is an invariant subspace of $\mS_t(s)$ for $s\ge 0$ and  \eqref{s8a}, \eqref{s5x}, and \eqref{s7ax} imply that $\mS_t(s)\vert_\Y\in\ml(\Y)$.
\end{proof}
 
We now verify that  $\Y$ is {\it $\A(t)$-admissible} in the sense of \cite[Definition 4.5.3]{Pazy} for $t\in [0,T]$; that is, $\Y$ is an invariant subspace of $\mS_t(s)$ as just shown in Lemma~\ref{L1}, and the restriction $\big(\mS_t(s)\vert_{\Y}\big)_{s\ge 0}$  is a strongly continuous semigroup on $\Y$. Regarding its generator we recall that the part $\A_{\Y}(t)$ of $\A(t)$ in $\Y$ is defined as 
$$
\dom(\A_{\Y}(t)):=\big\{\phi\in\dom(\A(t))\cap\Y\,;\, \A(t)\phi\in \Y\big\} \,,\quad \A_\Y(t)\phi:=\A(t)\phi\,.
$$
With these notations we have:

\begin{cor}\label{C1}
Let $t\in [0,T]$. The restriction  $\big(\mS_t(s)\vert_{\Y}\big)_{s\ge 0}$  is a strongly continuous semigroup on~$\Y$ with generator $\A_{\Y}(t)$. Moreover, $\big(\mS_t(s)\vert_{\E_1}\big)_{s\ge 0}$  is a strongly continuous semigroup on $\E_1$. 
\end{cor}

\begin{proof}
\noindent {\bf (i)} It readily follows from Lemma~\ref{L1} and Proposition~\ref{T1} that $\big(\mS_t(s)\vert_{\E_1}\big)_{s\ge 0}$  is a semigroup on $\E_1$. For the strong continuity
\begin{align}\label{s22}
\|\mS_t(s) \psi - \psi \|_{\E_1}\to 0\ \text{ as }\ s\to 0
\end{align}
for $\psi\in\E_1$ we use that
$$
U_{A(t)}\in C\big(J_\Delta,\ml_s(E_1)\big)\,,\qquad J_\Delta:=\{(a,\sigma)\in J\times J\,;\, \sigma\le a\}\,,
$$ according to \cite[Theorem~II.4.4.1, Corollary~II.4.4.2]{LQPP} and \eqref{A1a}, and prove \eqref{s22} analogously  to the strong continuity in $\E_0$ as in \cite[Theorem 4]{WebbSpringer} (using \eqref{A2} for~\mbox{$\ell=1$}).\\

\noindent {\bf (ii)} It follows from Lemma~\ref{L1} and Proposition~\ref{T1} that $\mS_t(s)\vert_{\Y}\in\ml(\Y)$ for $s\ge 0$ so that we only have to check the strong continuity. To this end, let $\psi\in\Y$. Then
\begin{align}\label{s20}
\|\mS_t(s)\psi-\psi\|_{\Y}&=\|\partial_a\mS_t(s)\psi-\partial_a\psi\|_{\E_0}+\|\mS_t(s)\psi-\psi\|_{\E_1}\,.
\end{align}
For the first term on the right-hand side of \eqref{s20} we use \eqref{s8a}, \eqref{s5}, and \eqref{s7ax} to get
\begin{align}\label{s21}
\|\partial_a\mS_t(s)\psi-\partial_a\psi\|_{\E_0}&\le\|A(t,\cdot) \mS_t(s) \psi -A(t,\cdot) \psi \|_{\E_0}+\| \mS_t(s)(\A(t)\psi)-\A(t)\psi \|_{\E_0}\nonumber\\
&\le\|A(t,\cdot)\|_{L_\infty(J,\ml(E_1,E_0))} \,\|\mS_t(s) \psi - \psi \|_{\E_1}+\| \mS_t(s)(\A(t)\psi)-\A(t)\psi \|_{\E_0}\,.
\end{align}
Since the second term on the right-hand side of \eqref{s21} goes to zero as $s\to 0$  due to the strong continuity of $(\mS_t(s))_{s\ge 0}$ in $\E_0$, it follows from \eqref{s22}, \eqref{s20}, and \eqref{s21} that $\mS_t(s)\psi\to \psi$ in $\Y$ as~$s\to 0$.
\end{proof}

Combining Corollary~\ref{C1} and Lemma~\ref{L1a} we deduce that $\Y$ is dense in $\E_1$.

\begin{cor}\label{C11}
$\Y$ is dense in $\E_1$.
\end{cor}

\begin{proof}
Let $\phi\in\E_1$ be arbitrary. Since  $\big(\mS_t(s)\vert_{\E_1}\big)_{s\ge 0}$  is a strongly continuous semigroup on $\E_1$ by Corollary~\ref{C1} (for fixed $t$), we have that
$$
\phi_\tau:=\frac{1}{\tau}\int_0^\tau \mS_t(s)\phi\,\rd s\to \phi\ \text{ in }\ \E_1 \ \text{ as }\ \tau\to 0
$$
with $\phi_\tau\in D(\A(t))$. The assertion now follows from the density of $\Y$ in $D(\A(t))$ proven in Lemma~\ref{L1a}.
\end{proof}

\section{Stability Properties}\label{Sec3}

A key ingredient in Kato's construction of an evolution system is the {\it stability} of the family of generators $(\A(t))_{t\in [0,T]}$ (see \cite[Definition 5.2.1]{Pazy}). We  prove this property in~$\E_0$ and in~$\E_1$ which turns out to be sufficient in our particular case. Actually, we  shall employ the notation of stability in the equivalent formulation (see \cite[Theorem 5.2.2]{Pazy}) of the associated family of semigroups $(\mS_t)_{t\in [0,T]}$ as stated in the next proposition.\\

For the remainder of this paper we also assume \eqref{A11}. That is, we impose \eqref{A1},~\eqref{A11}, \eqref{A2}, and~\eqref{A3} to hold from now on.

\begin{prop}\label{P2}
The family
$(\A(t))_{t\in [0,T]}$ is a stable family in $\E_\ell$ for $\ell\in\{0,1\}$ with constants $M_\ell$ and $\omega_\ell+\| b\|_\ell M_\ell$; that is,
\begin{equation}\label{stable0}
\bigg\|\prod_{j=1}^{n}\mS_{t_j}(s_j)\bigg\|_{\ml(\E_\ell)}\le M_\ell \exp\bigg\{\big(\omega_\ell+\| b\|_\ell M_\ell\big)\sum_{j=1}^{n}s_j\bigg\}\,,\quad s_j\ge 0\,,
\end{equation}
for any finite sequence $0\le t_1\le t_2\le\ldots\le t_n\le T$ and $n=1,2,\ldots $. 
\end{prop}

\begin{proof}
Let $\phi\in\E_\ell$.
Observe first from \eqref{100} that, for  $0\le t_1\le t_2\le T$, $s_1,s_2\ge 0$, and $a\in J$,
\begin{align*}
     \big[\mS_{t_2}(s_2)&\mS_{t_1}(s_1) \phi\big](a)\, =\, \left\{ \begin{aligned}
    &U_{A(t_2)}(a,a-s_2)\, \big[\mS_{t_1}(s_1)\phi\big](a-s_2)\, ,& &    s_2< a\, ,\\
    & U_{A(t_2)}(a,0)\, B_{\mS_{t_1}(s_1) \phi}^{t_2}(s_2-a)\, ,& &   a\le s_2\, ,
    \end{aligned}
   \right.\\[2mm]
	&=\, \left\{ \begin{aligned}
    &U_{A(t_2)}(a,a-s_2)\, U_{A(t_1)}(a-s_2,a-s_2-s_1)\,\phi(a-s_2-s_1)\, ,& &   s_1+s_2< a\, ,\\
    & U_{A(t_2)}(a,a-s_2)\,U_{A(t_1)}(a-s_2,0)\, B_{\phi}^{t_1}(s_1+s_2-a)\, ,& &   s_2< a\le s_1+ s_2\, ,\\
	& U_{A(t_2)}(a,0)\, B_{\mS_{t_1}(s_1)\phi}^{t_2}(s_2-a)\, ,& &    a\le s_2\,.
    \end{aligned}
 \right.
\end{align*}
In fact, for $n\ge 2$, $0\le t_1\le t_2\le\ldots\le t_n\le T$, $s_j\ge 0$, and $a\in J$ one shows inductively the  product formula
\begin{align*}
    & \bigg[\prod_{j=1}^{n}\mS_{t_j}(s_j) \phi\bigg](a)\\
	&=\, \left\{ \begin{aligned}
    &\prod_{j=1}^{n}U_{A(t_j)}\bigg(a-\sum_{l=j+1}^{n}s_l\,,\,a-\sum_{l=j}^{n}s_l\bigg)\,\phi\bigg(a-\sum_{l=1}^{n}s_l\bigg)\, ,& & \hspace{-3mm}  \sum_{l=1}^{n}s_l< a\, ,\\
    & \prod_{j=k}^{n}U_{A(t_j)}\bigg(a-\sum_{l=j+1}^{n}s_l\,,\,\bigg(a-\sum_{l=j}^{n}s_l\bigg)_+\bigg)\, B_{\prod_{i=1}^{k-1}\mS_{t_i}(s_i) \phi}^{t_k}\bigg(\sum_{l=k}^{n}s_l-a\bigg)\, ,& & \hspace{-3mm}  \sum_{l=k+1}^{n}s_l< a\le \sum_{l=k}^{n}s_l\, ,\\ &&&\quad k=1,\ldots, n\,,
    \end{aligned}
 \right.
\end{align*}
with the understanding that here and in the following
\begin{equation}\label{conv}
\sum_{l=n+1}^n s_l:=0\,,\qquad \prod_{i=1}^{0}\mS_{t_i}(s_i) \phi:=\phi\,.
\end{equation}
We verify \eqref{stable0} by induction with respect to $n\ge 1$ and establish simultaneously that 
\begin{equation}\label{stableB}
\left\| B_{\prod_{i=1}^{n-1}\mS_{t_i}(s_i) \phi}^{t_n}(s)\right\|_{E_\ell}\le \| b\|_\ell\,M_\ell \exp\bigg\{\big(\omega_\ell+\| b\|_\ell M_\ell\big)\bigg(s+\sum_{j=1}^{n-1}s_j\bigg)\bigg\}\,\|\phi\|_{\E_\ell}\,,\quad s\ge 0\,.
\end{equation}
To start with note from \eqref{100}, \eqref{A11}, and Gronwall's Inequality  (see \cite{WalkerIUMJ} or the computation below) that
$$
\|B_\phi^{t_1}(s)\|_{E_\ell}\le \|b\|_\ell\,M_\ell\, e^{(\omega_\ell +\|b\|_\ell M_\ell)s}\,\|\phi\|_{\E_\ell}\,,\quad  s\ge 0\,,
$$
and hence
$$
\|\mS_{t_1}(s_1)\|_{\ml(\E_\ell)}\le M_\ell\, e^{(\omega_\ell +\|b\|_\ell M_\ell)s_1}\,,\quad s_1\ge 0\,.
$$
That is, \eqref{stable0} and \eqref{stableB} hold true for $n=1$. We now  continue by induction for $n\ge 2$ and suppose  \eqref{stable0} and \eqref{stableB} to be true for $n-1$. To shorten notation set
$$
z(s):=B_{\prod_{i=1}^{n-1}\mS_{t_i}(s_i) \phi}^{t_n}(s)\,,\quad s\ge 0\,.
$$
Then, it follows from \eqref{500} and the above product formula for $\prod_{i=1}^{n-1}\mS_{t_i}(s_i) \phi$ that
\begin{align*}
& z(s)=\int_0^{s} \chi(a)\,b(a)\, U_{A(t_n)}(a,0)\, z(s-a)\,\rd a \\
&\qquad\quad +\int_{s}^{a_m} \chi(a)\,b(a)\, U_{A(t_n)}(a,a-s)\, \bigg[\prod_{i=1}^{n-1}\mS_{t_i}(s_i) \phi\bigg] (a-s)\,\rd a\\
&=\int_0^{s} \chi(a)\,b(a)\, U_{A(t_n)}(a,0)\, z(s-a)\,\rd a \\
&\quad+\sum_{k=1}^{n-1}\int_{s+\sum_{l=k+1}^{n-1}s_l}^{s+\sum_{l=k}^{n-1}s_l} \chi(a)\,b(a)\,U_{A(t_n)}(a,a-s)\, \prod_{j=k}^{n-1}U_{A(t_j)}\bigg(a-s-\sum_{l=j+1}^{n-1}s_l\,,\,\bigg(a-s-\sum_{l=j}^{n-1}s_l\bigg)_+\bigg)\\
&\qquad\qquad\qquad\qquad\qquad\qquad\qquad\qquad\qquad \times  B_{\prod_{i=1}^{k-1}\mS_{t_i}(s_i) \phi}^{t_k}\bigg(s+\sum_{l=k}^{n-1}s_l-a\bigg)\,\rd a\\
&\quad +\int_{s+\sum_{l=1}^{n-1}s_l}^{a_m} \chi(a)\,b(a)\, U_{A(t_n)}(a,a-s)\, \prod_{j=1}^{n-1}U_{A(t_j)}\bigg(a-s-\sum_{l=j+1}^{n-1}s_l\,,\,a-s-\sum_{l=j}^{n-1}s_l\bigg)\\
&\qquad\qquad\qquad\qquad\qquad\qquad\qquad\qquad\qquad \times \phi\bigg(a-s-\sum_{l=1}^{n-1}s_l\bigg)\,\rd a\,.
\end{align*}
Taking the norm in $E_\ell$  and using the stability assumption \eqref{A11} (and recalling convention~\eqref{conv})  yields
\begin{align*}
\| z(s)\|_{E_\ell}&
\le \|b\|_\ell\, M_\ell\int_0^{s} e^{\omega_\ell a}\, \| z(s-a)\|_{E_\ell}\,\rd a  \\
&\quad+\|b\|_\ell\, M_\ell \sum_{k=1}^{n-1}\int_{s+\sum_{l=k+1}^{n-1}s_l}^{s+\sum_{l=k}^{n-1}s_l} e^{\omega_\ell a}\, \bigg\| B_{\prod_{i=1}^{k-1}\mS_{t_i}(s_i) \phi}^{t_k}\bigg(s+\sum_{l=k}^{n-1}s_l-a\bigg)\bigg\|_{E_\ell}\,\rd a\\
&\quad +\|b\|_\ell\, M_\ell  \exp\bigg\{\omega_\ell\bigg(s+\sum_{j=1}^{n-1}s_j\bigg)\bigg\}\,\int_{s+\sum_{l=1}^{n-1}s_l}^{a_m}\bigg\| \phi\bigg(a-s-\sum_{l=1}^{n-1}s_l\bigg)\bigg\|_{E_\ell}\,\rd a\,.
\end{align*}
We next invoke \eqref{stableB} for $n-1$ to deduce
\begin{align*}
\| z(s)\|_{E_\ell}&
\le \|b\|_\ell\, M_\ell\int_0^{s} e^{\omega_\ell a}\, \| z(s-a)\|_{E_\ell}\,\rd a  \\
&\quad+\big(\|b\|_\ell\, M_\ell\big)^2 \sum_{k=1}^{n-1}\int_{s+\sum_{l=k+1}^{n-1}s_l}^{s+\sum_{l=k}^{n-1}s_l} e^{\omega_\ell a}\, \exp\bigg\{\big(\omega_\ell+\|b\|_\ell\, M_\ell\big)\bigg(s+\sum_{j=1}^{n-1}s_j-a\bigg)\bigg\}\,\rd a\, \|\phi\|_{\E_\ell}\\
&\quad +\|b\|_\ell\, M_\ell\,  \exp\bigg\{\omega_\ell\bigg(s+\sum_{j=1}^{n-1}s_j\bigg)\bigg\}\, \|\phi\|_{\E_\ell}\,.
\end{align*}
Computing  further we derive
\begin{align*}
\| z(s)\|_{E_\ell}&\le  \|b\|_\ell\, M_\ell\int_0^{s} e^{\omega_\ell a}\, \| z(s-a)\|_{E_\ell}\,\rd a  \\
&\quad+\big(\|b\|_\ell\, M_\ell\big)^2\,\exp\bigg\{\big(\omega_\ell+\|b\|_\ell M_\ell\big)\bigg(s+\sum_{j=1}^{n-1}s_j\bigg)\bigg\} \int_{s}^{s+\sum_{l=1}^{n-1}s_l} e^{-\|b\|_\ell M_\ell a}\, \rd a\, \|\phi\|_{\E_\ell}\\
&\quad +\|b\|_\ell\, M_\ell\,  \exp\bigg\{\omega_\ell\bigg(s+\sum_{j=1}^{n-1}s_j\bigg)\bigg\}\, \|\phi\|_{\E_\ell}\\
&= \|b\|_\ell\, M_\ell\int_0^{s} e^{\omega_\ell (s-a)}\, \| z(a)\|_{E_\ell}\,\rd a  \\
&\quad-\|b\|_\ell\, M_\ell\,\exp\bigg\{\big(\omega_\ell+\|b\|_\ell M_\ell\big)\bigg(s+\sum_{j=1}^{n-1}s_j\bigg)\bigg\} \, \bigg(e^{-\|b\|_\ell M_\ell a}\bigg\vert_{a=s}^{a=s+\sum_{l=1}^{n-1}s_l}\bigg)
\, \|\phi\|_{\E_\ell}\\
&\quad +\|b\|_\ell\, M_\ell\,  \exp\bigg\{\omega_\ell\bigg(s+\sum_{j=1}^{n-1}s_j\bigg)\bigg\}\, \|\phi\|_{\E_\ell}\,.
\end{align*}
Therefore, simplifying the last two terms we get
\begin{align*}
\| z(s)\|_{E_\ell}&
\le  \|b\|_\ell\, M_\ell\int_0^{s} e^{\omega_\ell (s-a)}\, \| z(a)\|_{E_\ell}\,\rd a  \\ 
&\quad
+\|b\|_\ell\, M_\ell\, e^{\omega_\ell s}\,\exp\bigg\{\big(\omega_\ell+\|b\|_\ell M_\ell\big)\sum_{j=1}^{n-1}s_j\bigg\} \,  \|\phi\|_{\E_\ell}
\end{align*}
for $s\ge 0$ and thus infer from Gronwall's Inequality that indeed
\begin{align*}
\left\| B_{\prod_{i=1}^{n-1}\mS_{t_i}(s_i) \phi}^{t_n}(s)\right\|_{E_\ell}& = \| z(s)\|_{E_\ell}\\
&\le \| b\|_\ell\,M_\ell \exp\bigg\{\big(\omega_\ell+\| b\|_\ell M_\ell\big)\bigg(s+\sum_{j=1}^{n-1}s_j\bigg)\bigg\}\,\|\phi\|_{\E_\ell}
\end{align*}
for $s\ge 0$ as claimed in~\eqref{stableB}. Next, we turn to~\eqref{stable0} still assuming it to be true for $n-1$. The above product formula for $\prod_{j=1}^{n}\mS_{t_j}(s_j) \phi$  yields
\begin{align*}
 \bigg\|\prod_{j=1}^{n}&\mS_{t_j}(s_j)\phi\bigg\|_{\E_\ell}
=\int_0^{a_m}\bigg\|\bigg[\prod_{j=1}^{n}\mS_{t_j}(s_j) \phi\bigg](a)\bigg\|_{E_\ell}\,\rd a\\
& \le\sum_{k=1}^{n}\int_{\sum_{l=k+1}^{n}s_l}^{\sum_{l=k}^{n}s_l}\bigg\| \prod_{j=k}^{n}U_{A(t_j)}\bigg(a-\sum_{l=j+1}^{n}s_l\,,\,\bigg(a-\sum_{l=j}^{n}s_l\bigg)_+\bigg)\bigg\|_{\ml(E_\ell)}\\
&\qquad\qquad\qquad\qquad\qquad\qquad\qquad\qquad\qquad \times   \bigg\|B_{\prod_{i=1}^{k-1}\mS_{t_i}(s_i) \phi}^{t_k}\bigg(\sum_{l=k}^{n}s_l-a\bigg)\bigg\|_{E_\ell}\,\rd a\\
&\quad +\int_{\sum_{l=1}^{n}s_l}^{a_m}\bigg\| \prod_{j=1}^{n}U_{A(t_j)}\bigg(a-\sum_{l=j+1}^{n}s_l\,,\,a-\sum_{l=j}^{n}s_l\bigg)\bigg\|_{\ml(E_\ell)}\,\bigg\|\phi\bigg(a-\sum_{l=1}^{n}s_l\bigg)\bigg\|_{E_\ell}\,\rd a\\
\end{align*}
if $\sum_{l=1}^{n}s_l<a_m$ (otherwise the last integral vanishes). We now first use~\eqref{A11} and then~\eqref{stableB} to get
\begin{align*}
\bigg\|\prod_{j=1}^{n}\mS_{t_j}(s_j)\phi\bigg\|_{\E_\ell}
&\le  M_\ell\, \sum_{k=1}^{n}\int_{\sum_{l=k+1}^{n}s_l}^{\sum_{l=k}^{n}s_l} e^{\omega_\ell a}  \, \bigg\|B_{\prod_{i=1}^{k-1}\mS_{t_i}(s_i) \phi}^{t_k}\bigg(\sum_{l=k}^{n}s_l-a\bigg)\bigg\|_{E_\ell}\,\rd a\\
&\qquad +M_\ell\,\exp\bigg\{\omega_\ell \sum_{l=1}^{n}s_l\bigg \}\,\|\phi\|_{\E_\ell}\\
&\le \|b\|_\ell\, M_\ell^2\, \sum_{k=1}^{n}\int_{\sum_{l=k+1}^{n}s_l}^{\sum_{l=k}^{n}s_l} e^{\omega_\ell a}  \, \exp\bigg\{\big(\omega_\ell+\|b\|_\ell M_\ell\big)\bigg(\sum_{i=1}^{n}s_i-a\bigg)\bigg\}\,\rd a\,\|\phi\|_{\E_\ell}\\
&\qquad +M_\ell\,\exp\bigg\{\omega_\ell \sum_{l=1}^{n}s_l\bigg \}\,\|\phi\|_{\E_\ell}\\
&= \|b\|_\ell\, M_\ell^2\, \exp\bigg\{\big(\omega_\ell+\|b\|_\ell M_\ell\big)\sum_{i=1}^{n}s_i\bigg\} \,\int_{0}^{\sum_{l=1}^{n}s_l} e^{-\|b\|_\ell\, M_\ell a} \,\rd a\,\|\phi\|_{\E_\ell}\\
&\qquad +M_\ell\,\exp\bigg\{\omega_\ell \sum_{l=1}^{n}s_l\bigg \}\,\|\phi\|_{\E_\ell}\\
&\le  M_\ell\, \exp\bigg\{\big(\omega_\ell+\|b\|_\ell M_\ell\big)\sum_{i=1}^{n}s_i\bigg\} \,\|\phi\|_{\E_\ell}\,.
\end{align*}
This is \eqref{stable0} and the proof of Proposition~\ref{P2} is complete.
\end{proof}


For the additional information stated in Corollary~\ref{C00} we shall also prove a similar estimate in~$L_p(J,E_\ell)$.

\begin{cor}\label{CP2}
Let $\ell\in\{0,1\}$. Given $p\in (1,\infty)$ there are $N\ge 1$ and $\xi\in\R$ (depending on $M_\ell$, $\omega_\ell$, $b$, and $p$)  such that 
\begin{equation}\label{stable0n}
\bigg\|\prod_{j=1}^{n}\mS_{t_j}(s_j)\phi\bigg\|_{L_p(J,E_\ell)}\le N \exp\bigg\{\xi \sum_{j=1}^{n}s_j\bigg\}\,\max\big\{\|\phi\|_{L_p(J,E_\ell)},\|\phi\|_{\E_\ell}\big\}\,,\quad s_j\ge 0\,,
\end{equation}
for every $\phi\in \E_\ell\cap L_p(J,E_\ell)$ and  any finite sequence $0\le t_1\le t_2\le\ldots\le t_n\le T$ and $n=1,2,\ldots $. 
\end{cor}

\begin{proof}
Let $\phi\in\E_\ell\cap L_p(J,E_\ell)$.  The product formula for $\prod_{j=1}^{n}\mS_{t_j}(s_j) \phi$  entails
\begin{align*}
 \bigg\|\prod_{j=1}^{n}&\mS_{t_j}(s_j)\phi\bigg\|_{L_p(J,E_\ell)}^p
=\int_0^{a_m}\bigg\|\bigg[\prod_{j=1}^{n}\mS_{t_j}(s_j) \phi\bigg](a)\bigg\|_{E_\ell}^p\,\rd a\\
& \le\sum_{k=1}^{n}\int_{\sum_{l=k+1}^{n}s_l}^{\sum_{l=k}^{n}s_l}\bigg\| \prod_{j=k}^{n}U_{A(t_j)}\bigg(a-\sum_{l=j+1}^{n}s_l\,,\,\bigg(a-\sum_{l=j}^{n}s_l\bigg)_+\bigg)\bigg\|_{\ml(E_\ell)}^p\\
&\qquad\qquad\qquad\qquad\qquad\qquad\qquad\qquad\qquad \times   \bigg\|B_{\prod_{i=1}^{k-1}\mS_{t_i}(s_i) \phi}^{t_k}\bigg(\sum_{l=k}^{n}s_l-a\bigg)\bigg\|_{E_\ell}^p\,\rd a\\
&\quad +\int_{\sum_{l=1}^{n}s_l}^{a_m}\bigg\| \prod_{j=1}^{n}U_{A(t_j)}\bigg(a-\sum_{l=j+1}^{n}s_l\,,\,a-\sum_{l=j}^{n}s_l\bigg)\bigg\|_{\ml(E_\ell)}^p\,\bigg\|\phi\bigg(a-\sum_{l=1}^{n}s_l\bigg)\bigg\|_{E_\ell}^p\,\rd a\\
\end{align*}
if $\sum_{l=1}^{n}s_l<a_m$, otherwise the last integral vanishes. Using~\eqref{A11} and~\eqref{stableB} we get as before
\begin{align*}
\bigg\|\prod_{j=1}^{n}\mS_{t_j}(s_j)\phi\bigg\|_{L_p(J,E_\ell)}^p
&\le  M_\ell^p\, \sum_{k=1}^{n}\int_{\sum_{l=k+1}^{n}s_l}^{\sum_{l=k}^{n}s_l} e^{p\omega_\ell a}  \, \bigg\|B_{\prod_{i=1}^{k-1}\mS_{t_i}(s_i) \phi}^{t_k}\bigg(\sum_{l=k}^{n}s_l-a\bigg)\bigg\|_{E_\ell}^p\,\rd a\\
&\qquad +M_\ell^p\,\exp\bigg\{p\omega_\ell \sum_{l=1}^{n}s_l\bigg \}\,\|\phi\|_{L_p(J,E_\ell)}^p\\
%
%
&\le \|b\|_\ell^p\, M_\ell^{2p}\, \exp\bigg\{p\big(\omega_\ell+\|b\|_\ell M_\ell\big)\sum_{i=1}^{n}s_i\bigg\} \,\int_{0}^{\sum_{l=1}^{n}s_l} e^{-p\|b\|_\ell\, M_\ell a} \,\rd a\,\|\phi\|_{\E_\ell}^p\\
&\qquad +M_\ell^p\,\exp\bigg\{p\omega_\ell \sum_{l=1}^{n}s_l\bigg \}\,\|\phi\|_{L_p(J,E_\ell)}^p\\
&= \frac{1}{p}\,\|b\|_\ell^{p-1}\, M_\ell^{2p-1}\, \exp\bigg\{p\big(\omega_\ell+\|b\|_\ell M_\ell\big)\sum_{i=1}^{n}s_i\bigg\}  \,\|\phi\|_{\E_\ell}^p\\
&\qquad -\frac{1}{p}\,\|b\|_\ell^{p-1}\, M_\ell^{2p-1} \,\,\exp\bigg\{p\omega_\ell \sum_{l=1}^{n}s_l\bigg \}\,\|\phi\|_{\E_\ell}^p\\
&\qquad +M_\ell^p\,\exp\bigg\{p\omega_\ell \sum_{l=1}^{n}s_l\bigg \}\,\|\phi\|_{L_p(J,E_\ell)}^p\,.
\end{align*}
This yields the claim.
\end{proof}

\section{Proof of Theorem~\ref{T2}} \label{Sec4}

Let us recall that we have verified in the previous sections under assumptions \eqref{A1},   \eqref{A11}, \eqref{A2}, and \eqref{A3} that $(\A(t))_{t\in [0,T]}$ is a stable family in $\E_0$ and in $\E_1$ (see Proposition~\ref{P2}), that $\Y$ is $\A(t)$-admissible (see Corollary~\ref{C1}), and that $\A\in C([0,T],\ml(\Y,\E_0))$ (see Lemma~\ref{L1a}). That is, we have verified all assumptions $(H_1)-(H_3)$ from \cite[Theorem 5.3.1]{Pazy} except for the stability of the family $(\A_\Y(t))_{t\in [0,T]}$  in $\Y$ which is fundamental for the construction of the evolution system (and part of assumption $(H_2)$ in \cite[Theorem 5.3.1]{Pazy}). As pointed out in the introduction,  $(\A_\Y(t))_{t\in [0,T]}$ need not be stable in $\Y$ in general. Yet, the proof of \cite[Theorem 5.3.1]{Pazy} still works in our case almost verbatim, where the assumption on the stability of the family $(\A_\Y(t))_{t\in [0,T]}$  in $\Y$ is replaced by the previously established stability in $\E_1$ together with the simple, but crucial, observation from \eqref{eq} that
\begin{equation}\label{key}
\big(\A(t_1)-\A(t_2)\big)\psi=\big(A(t_1,\cdot)-A(t_2,\cdot)\big)\psi\,,\quad \psi\in\Y\,,\quad t_1,t_2\in [0,T]\,,
\end{equation}
where the right-hand side is meaningful even if only $\psi\in\E_1$. 
For the sake of completeness and since we rely on the construction when proving additional properties later on, we provide the detailed construction of the evolution system but emphasize that it follows very much along the lines of the proof of \cite[Theorem~5.3.1]{Pazy}.

\subsection*{Construction of the Evolution System} 
Given $n\ge 1$ set $t_k^n:=(k/n)T$ for  $k=0,\ldots,n$ and let
\begin{equation}\label{3.3}
\A_n(t):=\A(t_k^n)\,,\quad t_k^n\le t <t_{k+1}^n\,,\quad k=0,\ldots,n-1\,,\qquad
\A_n(T):=\A(T)\,.
\end{equation}
Recall the notation $A(t)=A(t,\cdot)$ and set analogously
\begin{equation*} 
A_n(t):=A(t_k^n)\,,\quad t_k^n\le t <t_{k+1}^n\,,\quad k=0,\ldots,n-1\,,\qquad
A_n(T):=A(T)\,.
\end{equation*}
Since $A\in C\big([0,T],\ml(\E_1,\E_0)\big)$ by~\eqref{contA}, this definition yields
\begin{equation}\label{3.4}
\|A(t)-A_n(t)\|_{\ml(\E_1,\E_0)}\to 0\ \text{ as $n\to\infty$\,, uniformly in $t\in[0,T]$}\,.
\end{equation}
We then define the family of operators $(\U_n(t,s))_{0\le s\le t\le T}$ in $\ml(\E_0)\cap \ml(\E_1)$ by
\begin{equation}\label{3.5}
\begin{split}
 \U_n(t,s):=&\left\{\begin{array}{ll}
\mS_{t_j^n}(t-s)\,, & t_j^n\le s\le t\le t_{j+1}^n\,,\\
\mS_{t_k^n}\big(t-t_k^n\big)\,\Big[\prod_{j=l+1}^{k-1}\mS_{t_j^n}\left(\frac{T}{n}\right)\Big]\, \mS_{t_l^n}\big(t_{l+1}^n-s\big)\,, & k>l\,,\  t_k^n\le t\le  t_{k+1}^n\,,
\end{array}\right. \\
&\qquad \qquad \qquad \qquad\qquad \qquad \qquad\qquad \qquad \qquad \qquad\qquad \qquad t_l^n\le  s\le t_{l+1}^n\,.
\end{split}
\end{equation}
Then $(\U_n(t,s))_{0\le s\le t\le T}$ is an evolution system satisfying
\begin{equation}\label{3.6}
\U_n(s,s)=I\,,\qquad \U_n(t,s)=\U_n(t,r)\U_n(r,s)\,,\quad 0\le s\le r\le t\le T\,,
\end{equation}
and
\begin{equation}\label{3.7}
(t,s)\mapsto\U_n(t,s)\ \text{ is strongly continuous in $\E_0$ on }\ 0\le s\le t\le T\,.
\end{equation}
Moreover, it follows from Proposition~\ref{P2} that
\begin{align}\label{3.8}
\|\U_n(t,s)\|_{\ml(\E_\ell)}\le M_\ell e^{(\omega_\ell+M_\ell\|b\|_\ell)(t-s)}\,,\quad 0\le s\le t\le T\,,\quad \ell=0,1\,,
\end{align}
while Lemma~\ref{L1} ensures
\begin{align}\label{3.9}
\U_n(t,s)\Y\subset\Y\,,\quad 0\le s\le t\le T\,.
\end{align}
Since $\Y\subset D(\A(t))$ for $t\in [0,T]$ by Lemma~\ref{L1a}, the definition of $\U_n(t,s)$ shows for $\psi\in \Y$ that
\begin{align}
&\frac{\partial}{\partial t}\U_n(t,s)\psi =\A_n(t)\U_n(t,s)\psi\,,\qquad  t\not= t_j^n\,,\quad  j=0,1,\ldots , n\,,\label{3.10}\\
&\frac{\partial}{\partial s}\U_n(t,s)\psi=-\U_n(t,s)\A_n(s)\phi\,,\qquad  s\not=t_j^n\,,\quad j=0,1,\ldots , n\,.\label{3.11}
\end{align}
Fix $\psi\in\Y$ and consider $m,n\ge 1$ and $0\le s\le t\le T$. Then \eqref{3.10} and \eqref{3.11} imply that the map $r\mapsto \U_n(t,r)\U_m(r,s)\psi$ is differentiable with respect to $r\in [s,t]$ except for a finite number of values and, together with~\eqref{3.6}, entail that
\begin{align}
\U_n(t,s) \psi- \U_m(t,s)\psi&=\int_s^t \U_n(t,r)\,\big(\A_n(r)-\A_m(r)\big)\,\U_m(r,s)\psi\,\rd r\nonumber\\
&=\int_s^t \U_n(t,r)\,\big(A_n(r)-A_m(r)\big)\,\U_m(r,s)\psi\,\rd r \,,\label{3.12}
\end{align}
where we employ~\eqref{key} and~\eqref{3.9} for the second equality. Therefore, setting 
$$
\eta:=\max\{\omega_0+M_0\|b\|_0\,,\,\omega_1+M_1\|b\|_1\}
$$ 
we infer from~\eqref{3.12} and \eqref{3.8}  that
\begin{center}
\begin{align}
\|\U_n(t,s)& \psi- \U_m (t,s)\psi\|_{\E_0}\nonumber\\
&\le \int_s^t \|\U_n(t,r)\|_{\ml(\E_0)}\,\|A_n(r)-A_m(r)\|_{\ml(\E_1,\E_0)}\,\|\U_m(r,s)\|_{\ml(\E_1)}\,\|\psi\|_{\E_1} \,\rd r\nonumber\\
&\le M_0\,M_1\, e^{\eta(t-s)}\,\|\psi\|_{\E_1} \int_s^t \|A_n(r)-A_m(r)\|_{\ml(\E_1,\E_0)} \,\rd r\,.\label{3.13}
\end{align}
\end{center}
Thus, it follows from~\eqref{3.4} that $\U_n(t,s)\psi$ converges in $\E_0$, uniformly on $0\le s\le t\le T$, as $n\to\infty$. As  $\Y$ is dense in $\E_0$ by Lemma~\ref{L1a}, this convergence of
$\U_n(t,s)\psi$ together with~\eqref{3.8} imply that $\U_n(t,s)\psi$ converges  in $\E_0$ for every $\psi\in\E_0$, uniformly on $0\le s\le t\le T$, as $n\to\infty$.
Define now
\begin{align}\label{3.15}
\U_{\A}(t,s) \psi:=\lim_{n\to\infty}\U_n(t,s) \psi \,,\quad \psi\in\E_0\,,\quad 0\le s\le t\le T\,.
\end{align}
From~\eqref{3.6} and~\eqref{3.7} we see that $(\U_{\A}(t,s))_{0\le s\le t\le T}$ is an evolution system in $\E_0$ satisfying~\eqref{ES1} and~\eqref{ES2} while~\eqref{3.8} yields~\eqref{E1}.

\subsection*{Differentiability Properties} The proof of~\eqref{E2} and~\eqref{E3} is now exactly the same as in \cite[Theorem~5.3.1]{Pazy} again relying on~\eqref{key}: Fix $\psi\in\Y$, $n\ge 1$, $0\le s\le t\le T$ and $\tau\in [0,T]$. Then~\eqref{3.10} and~\eqref{3.11} imply that the map $r\mapsto \U_n(t,r)\mS_\tau(r-s)\psi$ is differentiable with respect to $r\in [s,t]$ except for a finite number of values and, together with~\eqref{3.6} and \eqref{key}, this entails that
\begin{align*}
\U_n(t,s) \psi- \mS_\tau(t-s)\psi&=\int_s^t \U_n(t,r)\,\big(\A_n(r)-\A(\tau)\big)\,\mS_\tau(r-s)\psi\,\rd r\nonumber\\
&=\int_s^t \U_n(t,r)\,\big(A_n(r)-A(\tau)\big)\,\mS_\tau(r-s)\psi\,\rd r 
\end{align*}
and therefore, using~\eqref{3.8},
\begin{align*} 
\|\U_n(t,s) \psi- \mS_\tau(t-s)\psi\|_{\E_0}
&\le M_0\,M_1\, e^{\eta(t-s)}\,\|\psi\|_{\E_1} \int_s^t \|A_n(r)-A(\tau)\|_{\ml(\E_1,\E_0)} \,\rd r\,.
\end{align*}
Passing to the limit as $n\to \infty$ gives
\begin{align}\label{3.17}
\|\U_{\A}(t,s) \psi- \mS_\tau(t-s)\psi\|_{\E_0}
&\le M_0\,M_1\, e^{\eta(t-s)}\,\|\psi\|_{\E_1} \int_s^t \|A(r)-A(\tau)\|_{\ml(\E_1,\E_0)} \,\rd r\,.
\end{align}
Taking $\tau=s$ in~\eqref{3.17} and using~\eqref{3.4} yields
\begin{align}\label{3.18}
\limsup_{t\searrow s}\frac{1}{t-s}\|\U_{\A}(t,s) \psi- \mS_s(t-s)\psi\|_{\E_0}=0
\end{align}
and the differentiability of $\mS_s(t-s)\psi$ from the right at $t=s$ then shows that also $\U_{\A}(t,s) \psi$ is differentiable from the right at $t=s$ with the same derivative. This entails~\eqref{E2}. Similarly, taking $\tau=t$ in~\eqref{3.17} and using~\eqref{3.4} yields
\begin{align}\label{3.19}
\limsup_{t\nearrow s}\frac{1}{t-s}\|\U_{\A}(t,s) \psi- \mS_t(t-s)\psi\|_{\E_0}=0
\end{align}
so that the differentiability of $\mS_t(t-s)\psi$ from the left at $s=t$ implies
\begin{align}\label{3.20}
\frac{\partial^-}{\partial s}\U_{\A}(t,s)\psi\big\vert_{s=t}=-\A(t)\psi\,.
\end{align}
If $s<t$, then \eqref{E2} and the strong continuity of $\U_{\A}(t,s)$ in $\E_0$ yield
\begin{align}
\frac{\partial^+}{\partial s}\U_{\A}(t,s)\psi&=\lim_{h\searrow 0}\frac{1}{h}\big(\U_{\A}(t,s+h)\psi-\U_{\A}(t,s)\psi\big)\nonumber\\
&=\lim_{h\searrow 0}\U_{\A}(t,s+h)\,\frac{1}{h}\big(\psi-\U_{\A}(s+h,s)\psi\big)=-\U_{\A}(t,s)\A(t)\psi\,.\label{3.21}
\end{align}
If $s\le t$, then \eqref{3.20} gives
\begin{align}
\frac{\partial^-}{\partial s}\U_{\A}(t,s)\psi&=\lim_{h\searrow 0}\frac{1}{h}\big(\U_{\A}(t,s)\psi-\U_{\A}(t,s-h)\psi\big)\nonumber\\
&=\lim_{h\searrow 0}\U_{\A}(t,s)\,\frac{1}{h}\big(\psi-\U_{\A}(s,s-h)\psi\big)=-\U_{\A}(t,s)\A(t)\psi\,.\label{3.22}
\end{align}
Hence, \eqref{3.21}-\eqref{3.22} entail \eqref{E3}.

\subsection*{Uniqueness} Let $(\V(t,s))_{0\le s\le t\le T}$ be an evolution system in $\E_0$ satisfying~\eqref{ES1}~-~\eqref{E3}. Since $\V(t,s)$ satisfies~\eqref{E3}, the construction of 
$\U_n(t,s)$ implies that the map $r\to \V(t,r)\U_{\A}(r,s)\psi$ is differentiable except for a finite number of values when $\psi\in\Y$ and
\begin{align*}
\V(t,s) \psi- \U_n(t,s)\psi&=\int_s^t \V(t,r)\,\big(\A(r)-\A_n(r)\big)\,\U_n(r,s)\psi\,\rd r\nonumber\\
&=\int_s^t \V(t,r)\,\big(A(r)-A_n(r)\big)\,\U_n(r,s)\psi\,\rd r 
\end{align*}
and hence
\begin{align*}
\|\V(t,s) \psi- \U_n(t,s)\psi\|_{\E_0}
&\le M_0\,M_1\, e^{\eta(t-s)}\,\|\psi\|_{\E_1} \int_s^t \|A(r)-A_n(r)\|_{\ml(\E_1,\E_0)} \,\rd r
\end{align*}
so that, passing to the limit as $n\to\infty$, gives $\V(t,s) \psi= \U_{\A}(t,s)\psi$ for every $\psi\in\Y$. Since $\Y$ is dense in $\E_0$, this is true for every $\psi\in\E_0$ and the uniqueness of the evolution system $(\U_{\A}(t,s))_{0\le s\le t\le T}$ follows. 

\subsection*{Strong Continuity in $\E_\alpha$}
Fix $\alpha\in [0,1)$ and recall that  \mbox{$\E_\alpha:=L_1(J,E_\alpha)$}
with  complex interpolation space $E_\alpha:=[E_0,E_1]_\alpha$. 
Let $0\le s\le t\le T$, $m,n\ge 1$, and consider $\psi\in \Y$. 
Since $[\E_0,\E_1]_\alpha =\E_\alpha$ according to \cite[Theorem~5.1.2]{BerghLoefstroem} we may interpolate the cases $\ell=0$ and $\ell=1$ in~\eqref{3.8} to deduce that there is $M_\alpha\ge 1$ (depending on $M_0$, $M_1$, and $\alpha$) with
\begin{align}\label{3.8x}
\|\U_n(t,s)\psi\|_{\E_\alpha}\le M_\alpha e^{\eta (t-s)}\,\|\psi\|_{\E_\alpha}\,.
\end{align}
Similarly, recalling from~\eqref{3.13} that
\begin{align}
\|\U_n(t,s)& \psi- \U_m (t,s)\psi\|_{\E_0}\le M_0\,M_1\, e^{\eta(t-s)}\,\|\psi\|_{\E_1} \int_s^t \|A_n(r)-A_m(r)\|_{\ml(\E_1,\E_0)} \,\rd r\label{3.13x}
\end{align}
and from \eqref{3.8} that
\begin{align}
\|\U_n(t,s)& \psi- \U_m (t,s)\psi\|_{\E_1}\le 2\,M_1\, e^{\eta(t-s)}\,\|\psi\|_{\E_1}\,, \label{3.13y}
\end{align}
 we may interpolate~\eqref{3.13x} and~\eqref{3.13y} to get
\begin{align}
\|\U_n(t,s)& \psi- \U_m (t,s)\psi\|_{\E_\alpha}\le M_4\, e^{\eta(t-s)}\,\|\psi\|_{\E_1} \left(\int_s^t \|A_n(r)-A_m(r)\|_{\ml(\E_1,\E_0)} \,\rd r\right)^{1-\alpha}\label{3.13z}
\end{align}
for some constant $M_4\ge 1$. It then follows from~\eqref{3.4} and~\eqref{3.13z} that $(\U_n(t,s) \psi)_{n\ge 1}$ converges in~$\E_\alpha$ while~\eqref{3.15} entails that this limit is necessarily $\U_\A(t,s) \psi$. Consequently, we deduce from~\eqref{3.8x} that
\begin{align}\label{3.8xy}
\|\U_\A(t,s)\psi\|_{\E_\alpha}=\lim_{n\to \infty}\|\U_n(t,s)\psi\|_{\E_\alpha}\le M_\alpha e^{\eta (t-s)}\,\|\psi\|_{\E_\alpha}\,.
\end{align}
Note that $\Y$ is dense in $\E_\alpha$ since $\Y$ is dense in $\E_1$ by Corollary~\ref{C11} and since $\E_1$ is dense in $\E_\alpha$. Thus, \eqref{3.8xy} is true for every $\psi\in \E_\alpha$ which proves~\eqref{E1x}. 

Finally, since
$[\E_0,\E_\beta]_{\alpha/\beta} =\E_\alpha$ for $0<\alpha<\beta$ according to \cite[Theorem~5.1.2]{BerghLoefstroem} and the reiteration theorem for the complex method \cite[Theorem~4.6.1]{BerghLoefstroem}, we have
$$
\|\U_\A(t,s)\psi-\U_\A(\tau,\sigma)\psi\|_{\E_\alpha}\le c\, \|\U_\A(t,s)\psi-\U_\A(\tau,\sigma)\psi\|_{\E_0}^{1-\alpha/\beta}\, \|\U_\A(t,s)\psi-\U_\A(\tau,\sigma)\psi\|_{\E_\beta}^{\alpha/\beta}
$$
for $\psi\in \E_1$ so that the strong continuity of $\U_{\A}(t,s)$ in $\E_\alpha$ is implied by~\eqref{E1x}, the strong continuity of $\U_{\A}(t,s)$ in $\E_0$ from~\eqref{ES2}, and the density of $\Y$ in $\E_\alpha$.

This completes the proof of Theorem~\ref{T2}.\qed
\vspace{2mm}

\section{Proof of Corollary~\ref{C00}} \label{Sec5}

Before proving Corollary~\ref{C00} we shall first establish the following embedding of $\Y$ (we do not assume $a_m<\infty$ in Lemma~\ref{L17}).

\begin{lem}\label{L17}
Consider $\sigma\ge 0$, $p\in (1,\infty)$, and $\theta\in [0,1)$ with $\sigma+\theta<1/p$. If $E_\theta=(E_0,E_1)_\theta$ is an arbitrary interpolation space, then 
\begin{align}\label{in7}
\Y \, \hookrightarrow\,  W_{p}^{\sigma}(J,E_{\theta})\,.
\end{align}
\end{lem}

\begin{proof}
Set $\ve:=1/p-\sigma-\theta>0$. We then infer from \cite[Equation~(3.6), Equation~(3.1)]{AmannGlasnik} (see also \cite[Corollary~4.3]{AmannGlasnik}) the Besov space embeddings
\begin{equation}\label{in1}
W_1^1(J,E_0)\hookrightarrow B_{1,\infty}^1(J,E_0)\hookrightarrow B_{1,1}^{1-\ve}(J,E_0)
\end{equation}
and
\begin{equation}\label{in2}
L_1(J,E_1)\hookrightarrow B_{1,\infty}^0(J,E_1)\hookrightarrow B_{1,1}^{-\ve}(J,E_1)\,.
\end{equation}
The definition of $\Y$ and real interpolation of \eqref{in1} and \eqref{in2}  yields \cite[Theorem~3.1, Corollary~4.3]{AmannGlasnik}
\begin{align}\label{in3}
\Y \, \hookrightarrow\, \big(L_1(J,E_1),W_1^1(J,E_0)\big)_{1-\theta,1}\,\hookrightarrow\,
&\, \big(B_{1,1}^{-\ve}(J,E_1),B_{1,1}^{1-\ve}(J,E_0)\big)_{1-\theta,1} \nonumber\\
\doteq\, & \, B_{1,1}^{1-\theta-\ve}\big(J,(E_1,E_0)_{1-\theta,1}\big) \,.
\end{align}
Since $\sigma=1/p-\theta-\ve<1-\theta-\ve$ we  deduce from \cite[Equation~(3.3)]{AmannGlasnik} and \cite[Corollary~4.3]{AmannGlasnik} that
\begin{align}\label{in4}
 B_{1,1}^{1-\theta-\ve}\big(J,(E_1,E_0)_{1-\theta,1}\big) \hookrightarrow  B_{p,1}^{\sigma}\big(J,(E_1,E_0)_{1-\theta,1}\big)\,.
\end{align}
Observe then from \cite[Equation~(3.2), Equation~(3.5)]{AmannGlasnik} and \cite[Corollary~4.3]{AmannGlasnik} that
\begin{align}\label{in5}
B_{p,1}^{\sigma}\big(J,(E_1,E_0)_{1-\theta,1}\big) \hookrightarrow  B_{p,p}^{\sigma}\big(J,(E_1,E_0)_{1-\theta,1}\big)\doteq W_{p}^{\sigma}\big(J,(E_1,E_0)_{1-\theta,1}\big)\,.
\end{align}
Finally, we gather \eqref{in3}-\eqref{in5} and use that the extremal property of the real interpolation method (see \cite[Theorem~3.9.1]{BerghLoefstroem}) implies the embedding
\begin{align*} 
(E_1,E_0)_{1-\theta,1}=(E_0,E_1)_{\theta,1}\hookrightarrow (E_0,E_1)_{\theta}=E_\theta
\end{align*}
 to deduce~\eqref{in7}.
\end{proof}

\subsection*{Proof of Corollary~\ref{C00}}

 Let now the assumptions of Corollary~\ref{C00} be true. Then, since $a_m<\infty$ and $p\in (1,\infty)$, we have 
$$
L_p(J,E_\ell)\hookrightarrow L_1(J,E_\ell)=\E_\ell\,,\quad \ell=0,1\,,
$$ 
so that  Corollary~\ref{CP2}  entails that there are $N\ge 1$ and $\xi\in\R$ with
\begin{equation*} 
\bigg\|\prod_{j=1}^{n}\mS_{t_j}(s_j)\bigg\|_{\ml(L_p(J,E_\ell))}\le N \exp\bigg\{\xi \sum_{j=1}^{n}s_j\bigg\}\,,\quad s_j\ge 0\,,
\end{equation*}
and therefore, interpolating the cases $\ell=0$ and $\ell=1$ and recalling that either $(\cdot,\cdot)_\theta=[\cdot,\cdot]_\theta$  or $(\cdot,\cdot)_\theta=(\cdot,\cdot)_{\theta,p}$ with $\theta\in [0,1]$, we deduce
\begin{equation}\label{app}
\bigg\|\prod_{j=1}^{n}\mS_{t_j}(s_j)\bigg\|_{\ml(L_p(J,E_\theta))}\le N \exp\bigg\{\xi \sum_{j=1}^{n}s_j\bigg\}\,,\quad s_j\ge 0\,,
\end{equation}
for  any finite sequence $0\le t_1\le t_2\le\ldots\le t_n\le T$ and $n=1,2,\ldots $. Let $\psi\in L_p(J,E_\theta)$.
We then infer from~\eqref{app} that the approximation $\U_n(t,s)$ in~\eqref{3.5} satisfies
\begin{equation} \label{compii}
\|\U_n(t,s)\psi\|_{L_p(J,E_\theta)}\le N e^{\xi(t-s)}\,\|\psi\|_{L_p(J,E_\theta)}\,,\quad n\ge 1\,,\quad 0\le s\le t\le T\,.
\end{equation}
Thus, the sequence $(\U_n(t,s)\psi)_{n\ge 1}$ is bounded in~$L_p(J,E_\theta)$ while it converges to $\U_{\A}(t,s)\psi$ in~$\E_0$ according to~\eqref{3.15}.
Since $L_p(J,E_\theta)$ is reflexive due to $p\in (1,\infty)$ and our assumption that  $E_\theta$ is reflexive, there is a subsequence $(\U_{n_k}(t,s)\psi)_{k\ge 1}$ that converges weakly in $L_p(J,E_{\theta})$ to $\U_{\A}(t,s)\psi$ by the Theorem of Eberlein-Smulyan, and hence, invoking~\eqref{compii},
$$
\|\U_{\A}(t,s)\psi\|_{L_p(J,E_{\theta})}\le\liminf_{k\to\infty}\|\U_{n_k}(t,s)\psi\|_{L_p(J,E_{\theta})}\le N e^{\xi(t-s)}\,\|\psi\|_{L_p(J,E_{\theta})}
$$
for $\psi\in L_p(J,E_\xi)$. This proves~\eqref{emb1r}. As for~\eqref{emb2r} we note that, given $\theta\in [0,1)$ there is $p\in (1,\infty)$ with $\theta<1/p$. Then 
$$
\Y\hookrightarrow L_p(J,E_{\theta})\hookrightarrow L_1(J,E_{\theta})=\E_\theta
$$
 due to Lemma~\ref{L17} and  $a_m<\infty$. Hence~\eqref{emb2r} is implied by~\eqref{emb1r}. This proves Corollary~\ref{C00}.\qed

\section{Proof of Theorem~\ref{T3} and Corollary~\ref{CT3}}\label{Sec6}

The quasilinear situation requires the following result on the continuous dependence of the evolution system on the operator $\A$ which is an easy consequence of its construction:

\begin{prop}\label{P1.6}
Suppose the two operators $A_1$ and $A_2$ satisfy \eqref{A1}, \eqref{A11}, \eqref{A2}, \eqref{A3} and denote the corresponding operators from Proposition~\ref{T1} by $\A_1$ and $\A_2$, respectively. Then, there are constants $R\ge 1$ and $\eta\in\R$ such that
 \begin{align*}
\|\U_{\A_1}(t,s)& \psi- \U_{\A_2}(t,s)\psi\|_{\E_0}
\le R\, e^{\eta(t-s)}\,\|\psi\|_{\E_1} \int_s^t \|A_1(\tau,\cdot)-A_2(\tau,\cdot)\|_{L_\infty(J,\ml(E_1,E_0))} \,\rd \tau
\end{align*}
for $0\le s\le t\le T$ and every $\psi\in\E_1$.
\end{prop}

\begin{proof}
Letting $\U_n^k(t,s)$ be the approximation of $\U_{\A_k}(t,s)$ from~\eqref{3.5} for $n\ge 1$ and $k=1,2$, we proceed as in the proof of Theorem~\ref{T2} by  noticing  that, given $\psi\in\Y$, the map $\tau\mapsto \U_n^1(t,\tau)\U_n^2(\tau,s)\psi$ is differentiable with respect to $\tau\in [s,t]$ except for a finite number of values and obtain
\begin{align*}
\U_n^1(t,s) \psi- \U_n^2(t,s)\psi
=\int_s^t \U_n^1(t,\tau)\,\big(A_{1,n}(\tau)-A_{2,n}(\tau)\big)\,\U_n^2(\tau,s)\psi\,\rd \tau \,. 
\end{align*}
This yields
\begin{align*}
\|\U_n^1(t,s)\psi- \U_n^2 (t,s)\psi\|_{\E_0}
\le M_0\,M_1\, e^{\eta(t-s)}\,\|\psi\|_{\E_1} \int_s^t \|A_{1,n}(\tau)-A_{2,n}(\tau)\|_{\ml(\E_1,\E_0)} \,\rd \tau\,.
\end{align*}
Since $\U_n^k (t,s)\psi$ converges to $\U_{\A_k} (t,s)\psi$ in $\E_0$ as $n\to\infty$ according to~\eqref{3.15} and since $A_{k,n}$ converges to $A_k$ in $C([0,T],\ml(\E_1,\E_0))$ as $n\to\infty$ by~\eqref{3.4}, the assertion follows from~\eqref{normA} for every $\psi\in\Y$. Since $\Y$ is dense in $\E_1$ by Corollary~\ref{C11}, the assertion is also true for every~$\psi\in\E_1$. 
\end{proof}

\subsection*{Proof of Theorem~\ref{T3}}

Suppose the assumptions of Theorem~\ref{T3}. It then readily follows for a continuous function $u: [0,T_0]\to \mathcal{B}$ with $T_0\in (0,T]$ that $t\mapsto \A\big(u(t),t\big)$ generates a unique evolution system $(\U_{\A(u)}(t,s)_{0\le s\le t\le T_0})$ in the sense of Theorem~\ref{T2}. Moreover, invoking Proposition~\ref{P1.6} and~\eqref{lipp} we obtain
 \begin{align*}
\|\U_{\A(u)}(t,s) \phi- \U_{\A(v)}(t,s)\phi\|_{\E_0}
&\le R\, e^{\eta(t-s)}\,\|\phi\|_{\E_1} \int_s^t \|A(u(\tau),\tau,\cdot)-A(v(\tau),\tau,\cdot)\|_{L_\infty(J,\ml(\E_1,\E_0))} \,\rd \tau\\
&\le L\,R\, e^{\eta(t-s)}\,\|\phi\|_{\E_1} \int_s^t \|u(\tau)-v(\tau)\|_{\E_0} \,\rd \tau
\end{align*}
for $0\le s\le t\le T$ and two continuous functions $u,v: [0,T_0]\to \mathcal{B}$. 
Theorem~\ref{T3} now follows by Banach's fixed point theorem exactly along the lines of the proof of \cite[Theorem~6.4.5]{Pazy} by showing that the mapping
\begin{equation}\label{fp}
u\mapsto \U_{\A(u)}(\cdot,0) \phi
\end{equation} 
has a unique fixed $u$ point in $\mathcal{S}:=\bar\B_{C([0,T_\phi],\E_0)}(\phi,r_0)$ provided $T_\phi\in (0,T]$ is small enough. Moreover, since $\phi\in\E_1\subset\E_\alpha$ for every $\alpha\in [0,1)$ we infer  $u\in C([0,T_\phi],\E_\alpha)$  from~\eqref{ES2alpha}.
This proves Theorem~\ref{T3}.
\qed

\subsection*{Proof of Corollary~\ref{CT3}}

Suppose the assumptions of Corollary~\ref{CT3}. Then $\phi\in\E_1\cap L_p(J,E_\theta)$ and we replace in the previous proof of  Theorem~\ref{T3} the set $\mathcal{S}$ by
$$
\mathcal{S}_0:=\big\{u\in \mathcal{S}\,;\,  u(t)\in L_p(J,E_\theta)\,,\, \|u(t)\|_{L_p(J,E_\theta)}\le (N_0+r_0)\|\phi\|_{L_p(J,E_\theta)}, \text{for $t\in [0,T_\phi]$}\big\}\,.
$$ 
Then $\mathcal{S}_0$ is closed in $C([0,T_\phi],\E_0)$ since $L_p(J,E_\theta)$ is reflexive under the assumptions of Corollary~\ref{CT3}. Moreover, using Corollary~\ref{C00},
$$
\|\U_{\A(u)}(t,0) \phi\|_{L_p(J,E_\theta)}\le N_0\, e^{\xi_0 T_\phi}\,\|\phi\|_{L_p(J,E_\theta)}\le(N_0+r_0)\,\|\phi\|_{L_p(J,E_\theta)}\,,\quad 0\le t\le T_\phi\,,
$$
for $u\in\mathcal{S}_0$ provided $T_\phi\in (0,T]$ is small enough. Now, as in the proof of  Theorem~\ref{T3}, the mapping~\eqref{fp} has a unique fixed point in $\mathcal{S}_0$.\qed

\bibliographystyle{siam}
\bibliography{AgeDiff_220203}

\end{document}